\documentclass[12pt]{amsart}
\usepackage{xcolor} 
\usepackage{graphicx}
\usepackage{enumerate}
\usepackage{mathabx}

\usepackage{amsmath}
\usepackage{amssymb}
\usepackage{amsrefs}

\DeclareMathSymbol\HH 0{AMSb}{`H}
\DeclareMathSymbol\I  0{AMSb}{`I}
\DeclareMathSymbol\R  0{AMSb}{`R}

\theoremstyle{plain}
\newtheorem{theorem}{Theorem}[section]
\newtheorem{lemma}[theorem]{Lemma}
\newtheorem{proposition}[theorem]{Proposition}
\newtheorem{corollary}[theorem]{Corollary}
\newtheorem{definition}[theorem]{Definition}

\theoremstyle{definition}

\newtheorem{question}{Question}[section]

\theoremstyle{remark}

\numberwithin{equation}{section}

\DeclareMathOperator{\acc}{acc}

\DeclareMathOperator{\dom}{dom}

\DeclareMathOperator{\sakne}{root}

\newcommand{\cprec}{\operatorname{<\!\!{\cdot}}}
\newcommand{\ttimes}{{\times}}

\title
{On the bounding, splitting, and distributivity numbers}
 
\author[A. Dow]{Alan Dow}
\address{University of North Carolina at Charlotte, 
Charlotte, NC 28223}
\email{adow@uncc.edu}
\author[S. Shelah]{Saharon Shelah}
\address{Department of Mathematics, Rutgers University, Hill Center,
 Piscataway, 
 New Jersey, U.S.A. 08854-8019}
\curraddr{Institute of Mathematics\\Hebrew University\\
Givat Ram, Jerusalem 91904, Israel}
\email{shelah@math.rutgers.edu}

\date{\today}

\thanks{
The research of the second  author was
 supported by
   the United States-Israel Binational Science Foundation (BSF Grant
   no. 2010405), and by the 
 NSF grant No. NSF-DMS 1101597.}



\begin{document}
\begin{abstract}
The cardinal invariants $  \mathfrak h, \mathfrak b,
 \mathfrak s$ of $\mathcal P (\omega)$ are known to satisfy
that $\omega_1 \leq \mathfrak h
 \leq\min\{\mathfrak b, \mathfrak s\}$. We prove
 that all inequalities can be strict. 
 We also introduce a new upper bound for $\mathfrak h$
and show that it can be less than $\mathfrak s$. 
 The key method is to utilize finite support
matrix iterations of ccc posets following \cite{BlassShelah}.
\end{abstract}
\maketitle

\bibliographystyle{plain}

\section{Introduction}

Of course the cardinal invariants of the continuum discussed in this
article are very well known (see 
\cite{vDHandbook}*{van Douwen, p111})
so we just give a brief reminder.  They deal with the mod finite 
ordering of the infinite subsets of the integers. We follow
convention and let $[\omega]^\omega$ (or $[\omega]^{\aleph_0}$)
denote the family of infinite subsets of $\omega$.
A set $A$ is a pseudo-intersection of a family
$\mathcal Y\subset 
[\omega]^\omega $ if $A$ is infinite and
$A\setminus Y$ is finite for all $Y\in \mathcal
Y$.  The family $\mathcal Y$ has the strong finite intersection property
 (sfip) if every finite subset has infinite intersection and
 $\mathfrak p$ is the minimum cardinal for which there is such
 a family with no pseudointersection. 
 A family $\mathcal I\subset \mathcal
 P(\omega)$ is an ideal if it is closed under finite unions and mod
 finite subsets. An ideal $\mathcal I \subset \mathcal P(\omega)$ is 
 dense if every  $Y\in [\omega]^\omega$ contains an infinite member
of $\mathcal I$. 
A set $S\subset \omega$
 is \textit{unsplit} by a family $\mathcal Y\subset [\omega]^\omega$ 
if $S$ is mod finite contained in one member of $\{ Y, \omega\setminus
Y\}$ for each $Y\in \mathcal Y$. 
The splitting number
 $\mathfrak s$ is the minimum cardinal of a family $\mathcal Y$ for
 which there is no infinite set unsplit by $\mathcal Y$ (i.e. 
every  $S\in [\omega]^\omega$  is \textit{split} by some member of
$\mathcal Y$ and $\mathcal Y$ is called a 
splitting family). 
The bounding number
 $\mathfrak b$ can easily be defined in these same terms, but it is
 best defined by the mod finite ordering, $<^*$, on the family of
 functions $\omega^\omega$. The cardinal $\mathfrak b$ is the minimum
 cardinal for which there is a $<^*$-unbounded  family $B
\subset \omega^\omega$ with $|B| = \mathfrak b$. 

The finite support iteration of the 
standard Hechler poset was shown in \cite{BaumDordal} to
 produce models 
 of $\aleph_1 = \mathfrak s < \mathfrak b$. 
 The consistency of $\aleph_1 = \mathfrak b <
 \mathfrak s=\aleph_2$ was established
  in \cite{Boulder}  with a countable support iteration
  of a special poset we now call $\mathcal Q_{Bould}$.
It is shown in \cite{FischerSteprans} that
 one can use Cohen forcing to select ccc subposets of 
 $\mathcal Q_{Bould}$ and finite support iterations
 to obtain  models of
 $\aleph_1< \mathfrak b   < \mathfrak s = \mathfrak b^+$. 
This result was improved in   \cite{BrendleFischer}
to show that the gap between $\mathfrak b$
and $\mathfrak s$ can be made arbitrarily large. 
The papers
\cite{BlassShelah} and \cite{BrendleFischer}
are able to use  ccc versions of the well-known
Mathias forcing in their iterations in place of
those  discovered in \cite{FischerSteprans}. 
The paper \cite{BrendleFischer} also nicely expands on the
method  of matrix iterated forcing first introduced
in \cite{BlassShelah}, as do a number of more recent
papers (see \cites{Fischer18, Diego13}
and 
\cite{Fischer17} using template forcing).
 The distributivity number (degree)
 $\mathfrak h$ was first studied in \cite{BPS80}. It equals
  the minimum number of dense ideals whose
 intersection is simply the Fr\'echet ideal $[\omega]^{<\omega}$. 
 It was shown in \cite{BPS80}, that 
 $\mathfrak p\leq \mathfrak h\leq \min\{\mathfrak b, \mathfrak s\}$.
 Our goal is to separate all these cardinals. 
 We succeed but confront  a new problem since we   
 use the result, also from \cite{BPS80}, 
 that $\mathfrak h\leq \operatorname{cf}(\mathfrak c)$.

\section{A new bound on $\mathfrak h$}
In \cite{BPS80}, a family $\mathfrak A$ of maximal almost disjoint families
of infinite subsets of $\omega$ is called a matrix. A matrix $\mathfrak A$
is   \textit{shattering\,} if the entire collection $\bigcup\mathfrak A$ is splitting. 
Evidently, if $\{ s_\alpha : \alpha < \kappa\}$ is  a splitting family,
then the family $\mathfrak A = \{ \{s_\alpha,\omega\setminus s_\alpha \} : 
\alpha <\kappa\}$ is a shattering matrix.   A shattering matrix
 $\mathfrak A = \{ \mathcal A_\alpha : \alpha <\kappa\}
 $ is \textit{refining\/}, if for all $\alpha < \beta <\kappa$,
  $\mathcal A_\beta$ refines $ \mathcal A_\alpha$ in the natural sense that 
  each member of $\mathcal A_\beta$ is mod finite contained in
  some member of $\mathcal A_\alpha$. Finally, a \textit{base matrix} 
  is a refining shattering matrix $\mathfrak A$ satisfying
  that $\bigcup\mathfrak A$ is dense in $(\mathcal P(\omega)
  /\operatorname{fin}, \subset^*)$ (i.e. a $\pi$-base for $\omega^*$).

We add condition (6) to the following result from \cite{BPS80}.

\begin{lemma} The\label{newh}
 value of $\mathfrak h$ is the least cardinal 
$\kappa$ such that any of the following hold:
\begin{enumerate}
\item the Boolean algebra $\mathcal P(\omega)/\operatorname{fin}$ is not 
$\kappa$-distributive,
\item there is a shattering matrix of cardinality $\kappa$,
\item there is a shattering and refining matrix indexed by $\kappa$,
\item there is a base matrix of cardinality $\kappa$,
\item there is a family of $\kappa$ many nowhere dense subsets of $\omega^*$ whose union is dense,
\item there is a sequence $\{\mathcal S_\alpha : \alpha <\kappa\}$ of splitting families satisfying that no 1-to-1 selection $\langle s_\alpha : \alpha\in \kappa\rangle \in \Pi\{\mathcal S_\alpha : \alpha\in\kappa\}$ has a pseudo-intersection.  
\end{enumerate} 
\end{lemma}

\begin{proof}
 Since (1)-(5) are proven  in \cite{BPS80}, it is sufficient
 to prove that, for a cardinal $\kappa$,  (3) and (6) are equivalent. 
 First suppose that $\mathfrak A = \{\mathcal A_\alpha : \alpha < \kappa\}$ is a 
 refining and shattering matrix. Since the matrix is refining, it follows easily
 that, 
 for each $\alpha<\kappa$,  $\{ \mathcal A_\beta : \alpha \leq \beta < \kappa\}$ 
 is a shattering matrix.  Therefore,
for each $\alpha<\kappa$, 
  $\mathcal S_\alpha=
 \bigcup \{ \mathcal A_\beta : \alpha \leq \beta\}$ is a splitting family.
 Similarly, the refining property ensures that
if $\langle a_\alpha : \alpha \in\kappa\rangle \in \Pi\{\mathcal S_\alpha : \alpha \in
 \kappa\}$, then $\{ a_\alpha :\alpha\in\kappa\}$ has no pseudo-intersection. 
 
 Now assume that $\{\mathcal S_\alpha : \alpha <\kappa\}$ is a sequence
 of splitting families as in (6). By \cite{BPS80}, it is sufficient to prove
 that $\mathfrak h\leq \kappa$, so let us assume that $\kappa\leq\mathfrak h$. 
 We now make an observation about $\kappa$:  for each infinite
  $b\subset\omega$, $\alpha<\kappa$ 
  and family $\mathcal S' \subset[\omega]^\omega$ of cardinality
  less than $\kappa$, 
  there is  an infinite $a\subset b$ and 
  an $s \in \mathcal S_\alpha
  \setminus \mathcal S'$ such that $a\subset s$ and
  $s $ splits $b$.
 We prove this claim.   
 We may ignore all members of $\mathcal S'$ that
  are mod finite disjoint, or mod finite include, $b$.
 Since the family $\{ \{ s'\cap b, b\setminus s'\} : 
 s'\in \mathcal S'\} $ is not shattering (as a family of subsets
 of $b$) there is an infinite $b'\subset b$ that is not split by $\mathcal S'$. 
 Choose any $s\in \mathcal S_\alpha$ that splits $b'$ 
 and let $a = s\cap b'$. Evidently,
  $s$ also splits $b$.  Since the ideal generated by a splitting family is dense, 
  we may choose a maximal almost disjoint family
   $\mathcal A_0$ contained in the ideal generated by $\mathcal S_0$.
 Let $s_0 $ denote any mapping from $\mathcal A_0$ 
  into $\mathcal S_0$ satisfying that $a\subset s_0(a)$ for all
   $a\in \mathcal A_0$. Suppose that $\alpha <\kappa$ and that
   we have chosen a refining 
   sequence $\{\mathcal A_\gamma : \gamma<\alpha\}$ 
   of maximal almost disjoint families together with mappings
    $\{s_\gamma : \gamma<\alpha\}$ so that for each
     $a\in \mathcal A_\gamma $ ,
    $a\subset s_\gamma(a)\in \mathcal S_\gamma$.
 The extra induction assumption is that for all $a\in \mathcal A_\gamma$,
   $s_\gamma(a)$ is not an element of
    $\{ s_\beta (a') : \beta <\gamma \ \mbox{and} \ 
    a\subset^* a'\in \mathcal A_\beta\}$. The existence of the family
     $\mathcal A_\alpha$ and the mapping $s_\alpha$ 
     satisfying the induction conditions easily
     follows from the above Observation. Now we verify
     that $\mathfrak A = \{ \mathcal A_\alpha : \alpha < \kappa\}$ 
     satisfies that $\bigcup\mathfrak A$ is splitting.
   Fix any infinite $b\subset\omega$
   and choose $a_\alpha\in \mathcal A_\alpha$, for each $\alpha\in\kappa$
   so that $b\cap a_\alpha$ is infinite. By construction,
    $\{ s_\alpha(a_\alpha) : \alpha\in \kappa\}$ is a 1-to-1 selection
    from $\Pi\{\mathcal S_\alpha : \alpha\in \kappa\}$. Since $b$ is therefore
    not a pseudo-intersection, there is an $\alpha<\kappa$ such
    that $b\setminus s_\alpha(a_\alpha)\subset
    b\setminus a_\alpha$ is infinite. 
\end{proof}

The following is an immediate corollary to condition (6) in Lemma \ref{newh}
and provide two approaches to bounding the value of $\mathfrak h$.
 
\begin{corollary} [\cites{BPS80,Bla89}]\label{2.2}
\begin{enumerate}
\item 
 $\mathfrak h\leq\operatorname{cf}(\mathfrak c)$.
\item A poset $\mathbb P$  forces that $\mathfrak h\leq \kappa$  if
 $\mathbb P$ preserves 
  $\kappa$ and   can be written as an increasing chain 
 $\{\mathbb P_\alpha :\alpha < \kappa\}$ of completely embedded posets
 satisfying that each $\mathbb P_{\alpha+1}$ adds a real not added by
  $\mathbb P_{\alpha}$.
  \end{enumerate}
\end{corollary}
\begin{proof}
  For the statement in (1),
  let  $\{\kappa_\alpha :\alpha <  \operatorname{cf}(\mathfrak c)\}$
 be increasing and cofinal in $\mathfrak{c}$.
Let $\{ x_\xi : \xi \in \mathfrak c\}$ be an enumeration of 
  $[\omega]^{\aleph_0}$. To apply (6) from Lemma \ref{newh},
let $\mathcal S_\alpha =
\{ x_\xi :(\forall\eta<\kappa_\alpha)~ x_\eta\not\subset^* 
x_\xi \}$. Since every infinite $Y\subset\omega$ can be
refined by an almost disjoint family of cardinality $\mathfrak c$,
 it follows that $\mathcal S_\alpha$ is splitting.
   For the statement in (2), let $G$
  be a $\mathbb P$-generic filter and, for each $\alpha\in\kappa$,
   let $G_\alpha = G\cap \mathbb P_\alpha$. To apply (6),
    let $\mathcal S_\alpha$ be the set of $x\in [\omega]^{\aleph_0}$ that
    contain no
    infinite $y\in V[G_\alpha]$. To see that $\mathcal S_\alpha$ is splitting
     in either case,  given any infinite $x\subset \omega$,
    consider an enumeration $\{ x_t : t\in 2^{<\omega}\}$.
    Then, for all $\alpha\in \kappa$, there is an $f_\alpha \in 2^\omega$
    so that $\{ x_{f_\alpha\restriction n} : n\in \omega\}\in \mathcal S_\alpha$.     
\end{proof}
Our introduction of condition (6) in Lemma \ref{newh} is motivated
by the fact that
it provides us with a new approach to bounding $\mathfrak h$.
We introduce the following variant of condition (6) in Lemma \ref{newh}
and note that a shattering refining matrix will
 fail to satisfy the second condition.

\begin{definition}
 Let $\kappa <\lambda$ be cardinals\label{kappalambda}
  and say that a family
  $\{ x_\alpha : \alpha < \lambda\}$ of infinite subsets of $\omega$
   is $(\kappa,\lambda)$-shattering
  if, for all infinite $b\subset \omega$
\begin{enumerate}
\item the set $\{\alpha <\lambda: b\subset^* x_\alpha\}$
  has cardinality less than
 $\kappa$, 
  and
  \item the set $\{  \alpha < \lambda : b\cap x_\alpha =^*\emptyset\}$ 
  has cardinality less than $\lambda$.
\end{enumerate}
Say that a $(\kappa,\lambda)$-shattering
family is strongly
 $(\kappa,\lambda)$-shattering if it contains no splitting
 family of size less than $\lambda$.
\end{definition}
 
 Needless to say a $(\kappa,\lambda)$-shattering family
 is strongly $(\kappa,\lambda)$-shattering if $\lambda
 =\mathfrak s$ and this is the kind of families we are interested in.
  However it seems likely that producing strongly $(\kappa,\lambda)$-shattering
  families would be interesting (and as difficult)
  even without requiring
  that $\lambda=\mathfrak s$. 
  Nevertheless $\mathfrak s$ is necessarily
  bounded by  $\lambda$ as we show next.

\begin{proposition}
 If there is a $(\kappa,\lambda)$-shattering family, then 
  $\mathfrak h \leq \kappa$ and $\mathfrak s \leq \lambda$. 
\end{proposition}

\begin{proof}
 Let $\mathcal S = \{ x_\alpha : \alpha <\lambda\}$ be a
  $(\kappa,\lambda)$-shattering family. Given any infinite
   $b\subset\omega$, there is a $\beta<\lambda$ such that
   each of $b\subset^* x_\beta$ and $b\cap x_\beta=^*\emptyset$ fail.
    This means that $\mathcal S$ is splitting. By condition (1) in
 Definition \ref{kappalambda} and 
 applying condition (6) of Lemma \ref{newh} 
 with $\mathcal S_\alpha = \mathcal S$ 
 for all $\alpha<\kappa$, it follows that $\mathfrak h\leq \kappa$.
\end{proof}

   For any index set $I$  the standard poset for adding Cohen reals,
    $\mathcal C_I$, is the set of all finite functions into $2$ with 
    domain a finite subset of 
    $I$ where $p<q$ providing $p\supset q$. If $\lambda$ is an ordinal,
     then we may use $\dot x_\alpha$ to be the canonical 
     $\mathcal C_\lambda$-name $\{ (\check n, 
     \{\langle\alpha{+} n,1\rangle\} : n\in \omega\}$ (i.e.
      for $s\in\mathcal C_{\lambda}$, $s\Vdash n\in \dot x_\alpha$
      providing $s(\alpha+n)=1$).

      It is routine to verify that, for any regular cardinal $\lambda
      >\aleph_1$,
       forcing with $\mathcal C_{\lambda}$ will naturally add
       an $(\aleph_1,\lambda)$-shattering family
       but is is clear that 
       this  family
       would not be strongly $(\aleph_1,\lambda)$-shattering
       because it has a splitting subfamily of
       cardinality $\aleph_1$.
Nevertheless, it may be possible with further forcing,
to  have it become strongly $(\kappa,\lambda)$-shattering
for some $\aleph_1\leq \kappa<\mathfrak s$.

In Theorem \ref{noth} we will prove that it is consistent 
with $\aleph_2 < \kappa^+ < \mathfrak c$ 
that there is a strongly $(\kappa,\kappa^+)$-shattering family.  

\begin{question} Assume that $\kappa <\lambda$ are regular
cardinals and that there is a strongly 
$(\kappa,\lambda)$-shattering family. We
pose the following questions.
\begin{enumerate}
 \item Is it consistent that $\kappa^+<\lambda$?
 \item Is it consistent that $\lambda<\mathfrak b$?
 \item Is it consistent that $\kappa < \mathfrak b < \lambda$?
\end{enumerate}
\end{question}

  \section{Matrix forcing and distinguishing $\mathfrak h, \mathfrak s,
  \mathfrak b$}

  In this section we recall the forcing methods for distinguishing
   $\mathfrak b$ and $\mathfrak s$ and apply them to prove
   the main results. We denote by $\mathbb D$
   the standard (Hechler) poset for adding a dominating 
   real. The poset $\mathbb D$ is an ordering on
    $\omega^{<\omega}\times \omega^\omega$ where
     $(s,f)<(t,g)$ providing
       $g\leq f$ and $s$ extends $t$ by values that
       are coordinatewise above $g$.
Given a sfip family  $\mathcal F$ of subsets
of $\omega$, there are
two main posets for adding a pseudo-intersection. The
Mathias-Prikry style poset is  $\mathbb M({\mathcal F})$
 that
 consists  
 of pairs $(a,A)$
 where 
 and $A  $ is in the filter
 base generated by $ \mathcal F$,
  $a\subset\min(A)$,  and  $\mathbb M(\mathcal F)$ is
  ordered by 
 $(a_1,A_1) < (a_2,A_2)$ providing
  $a_2\subset a_1\subset a_2\cup A_2$
   and $A_1\subset A_2$.  When the context is clear, we will
   let $\dot x_{\mathcal F}$ denote the canonical name,  
    $\{ (\check n, (a,\omega\setminus n{+}1)): n\in a\subset n+1\}$,
    which is forced to be the desired pseudo-intersection.   
        When $\mathcal U$ is a free ultrafilter
   on $\omega$, 
   $\mathbb M({\mathcal U})$ was the poset used
   in \cite{BlassShelah} and \cite{BrendleFischer} and, in
   this case $\dot x_{\mathcal U}$ is unsplit by the set of 
   ground model subsets of $\omega$.   When mixed with matrix iteration 
   methods, the ultrafilter $\mathcal U$ can be constructed so
   as to not add a dominating real. 
   
   The Laver style poset, $\mathbb L(\mathcal F)$, is also very useful
   in matrix iterations and is defined as follows. The members of
    $\mathbb L(\mathcal F)$ are   subtrees $T$ of
     $\omega^{<\omega}$ with a root or  stem, $\sakne(T)$, and for all
      $\sakne(T)\subseteq t\in T$, the set $\operatorname{Br}(T,t) = 
       \{ j \in\omega : t^\frown j\in T\}$ is an element of
       the filter generated by  $\mathcal F$. 
       This poset is ordered by $\subset$. 
  For each $T\in \mathbb L(\mathcal F)$ and $t\in T$,
   the subtree $T_t = \{ t'\in T : t\cup t' \in \omega^{<\omega}
   \}$ is also a condition.      
       The generic function, $\dot f_{\mathbb L(\mathcal F)}$,
       added by $\mathbb L(\mathcal F)$ can be described 
   by
       the name of the  union of the branch of $\omega^{<\omega}$
       named by $  \{ (\check{t}, \left(\omega^{<\omega}\right)_t ) :
        t\in \omega^{<\omega}\}$.  This poset forces
        that $\dot f_{\mathbb L(\mathcal F)}$ dominates the
        ground model reals and the range of $\dot f_{\mathbb L(\mathcal F)}$ 
        is a pseudo-intersection of $\mathcal F$. Again,
        if $\mathcal F$ is an ultrafilter, this pseudo-intersection
        is not split by any ground model set. 
        
      For each sfip family 
      $\mathcal U$ on $\omega$,
        each of the posets $\mathbb D$, $\mathbb M(\mathcal U)$,
        and $\mathbb L(\mathcal U)$ is $\sigma$-centered. We just
        need this for the fact that this ensures that
        they are upwards ccc.
   
   For a poset $P$ and a set $X$, a canonical $P$-name for a subset
   of $X$ will be a name of the form $\bigcup\{\check{x}\ttimes A_x : 
   x\in X\}$ where, for each $x\in X$, $A_x$ is an antichain of $P$. 
   An antichain of $P$ is a set whose elements are pairwise incompatible
   and a subset of $P$ is predense if its downward closure is dense.
The incompatibility relation on $P$ is denoted as $\perp_P$.   Of
   course if $\dot Y$ is any $P$-name of a subset of $X$, there is 
   a canonical name that is forced to equal it. If $P$ is ccc and
    $X$ is countable, then the set of canonical $P$-names for subsets
    of $X$ has cardinality at most $|P|^{\aleph_0}$. When
    we say that a poset $P$ forces a statement, we intend the meaning
    that every element (i.e. $1_P$) of $P$ forces that statement.
   
   Recall that a poset $P$ is a complete suborder of a poset
    $Q$ providing $P\subset Q$, $<_P\subset <_Q$, $\perp_P\subset
    \perp_Q$, and every predense subset of $P$ is predense in $Q$.
    We write $P\cprec Q$ to mean that $P$ is a complete
    suborder of $Q$.     If $G$ is a $Q$-generic filter and if $P
    \cprec Q$, then $G\cap P$ is a $P$-generic filter. 
    If we say that $Q$ forces some property 
    concerning the forcing extension
    by $P$, we mean that for each $Q$-generic filter
     $G$,      that property holds in $V[G\cap P]$.

  We say that $p\in P$ is a reduct (or a $P$-reduct)
    of $q\in Q$
    if every $r\leq p$ in $P$ is compatible with $q$ in $Q$.  If $P\cprec Q$,
     then every $q\in Q$ has a $P$-reduct.    
    If $\{ P_\alpha : \alpha <\delta\}$ is a
     $\cprec$-increasing chain of posets, then 
     the union $P_\delta = \bigcup\{P_\alpha : \alpha<\delta\}$ 
     satisfies that $P_\alpha\cprec P_\delta$ for all $\alpha<\delta$.
    Before we recall the definition of a matrix-iteration,
     we introduce the following generalization used in \cite{Shelahpseudo}.

\begin{definition}
Let $\kappa>\omega_1$ be a regular\label{matrix} cardinal. For
an\label{matrixposet}
   ordinal  $  \zeta$, a $\kappa\ttimes \zeta$-matrix
 of posets is a family $\{ P_{\alpha,\xi} : \alpha < \kappa, \xi<\zeta\}$ 
 of ccc posets satisfying, for each $\alpha<\kappa$,
  and $\xi < \eta<\zeta$:
\begin{enumerate}
 \item $P_{\alpha,\xi} \cprec P_{\beta,\xi}$ for all $\alpha<\beta <\kappa$,
 \item  $P_{\beta,\xi} = 
 \bigcup \{ P_{\eta,\xi} : \eta < \beta\}$ for $\beta < \kappa$ 
 with $\mathop{cf}(\beta)>\omega$,
 and
 \item for some $\gamma<\kappa$, $P_{\beta,\xi}\cprec P_{\beta,\eta}$
 for all $
\gamma\leq \beta<\kappa$.
\end{enumerate}
\end{definition}

\begin{lemma}
 If $\{ P_{\alpha,\xi} : \alpha < \kappa ,\xi<\zeta\}$ is a 
 $\kappa\ttimes \zeta$-matrix 
 of posets, then\label{limit}
 there is a sequence 
  $\{ P_{\kappa,\xi} : \xi \leq \zeta\}$ of
  ccc posets such that, for each $\xi<\eta\leq \zeta$:
 
\begin{enumerate}
 \item $P_{\kappa ,\xi} =\bigcup\{ P_{\alpha,\xi}:\alpha < \kappa \}$
 \item $P_{\kappa,\zeta} = \bigcup \{ P_{\kappa,\xi} : \xi < \zeta\}$,
  
 \item for\label{chain}
  all $\alpha<\kappa$, $P_{\alpha,\xi}\cprec P_{\kappa,\xi}$, and
 \item $P_{\kappa,\xi}\cprec P_{\kappa,\eta} $\label{top}.  
\end{enumerate}
\end{lemma}

\begin{proof} Item (\ref{chain}) follows  
immediately from item (1) of Definition \ref{matrix}. 
  To  prove (\ref{top}) 
  it suffices to check 
  that $P_{\alpha,\xi} \cprec P_{\kappa,\eta}$ for all 
   $\alpha <\kappa$ and $\xi <\eta<\zeta$.
  Let $\alpha<\kappa$ and $\xi<\eta<\zeta$. 
  Choose 
$\gamma<\kappa$ as in property (3) of Definition \ref{matrix}.
Now
we have $P_{\alpha,\xi}\cprec P_{\gamma,\xi}\cprec  P_{\gamma,\eta}\cprec P_{\kappa,\eta}$. 
Since $\cprec$ is a transitive relation, the proof is complete.
\end{proof}

The terminology ``matrix iterations'' is used in  
 \cite{BrendleFischer}, see also forthcoming preprint 
(F1222) from the second author. 

\begin{definition}
For an infinite cardinal $\kappa$ with
uncountable cofinality, and an ordinal $\zeta$,
a $\kappa\ttimes\zeta$-matrix iteration is a family
$$\langle\langle \mathbb P_{\alpha,\xi} : \alpha\leq \kappa,
 \xi \leq \zeta\rangle,
  \langle \dot{\mathbb Q}_{\alpha,\xi} : 
  \alpha \leq \kappa,\xi<\zeta\rangle\rangle$$
  where, for each $\alpha<\beta\leq \kappa$
  and $\xi <\eta \leq \zeta$:
 
\begin{enumerate}
 \item $\mathbb P_{\beta,\xi}$ is a ccc poset,
 \item $\mathbb P_{\alpha,\xi}\cprec \mathbb P_{\beta,\xi}\cprec
 \mathbb P_{\beta,\eta}$,
  \item  $\mathbb P_{\kappa,\xi}$ is the union of the chain
   $\{ \mathbb P_{\gamma,\xi} : \gamma <\kappa\}$,
 \item $\dot{\mathbb Q}_{\alpha,\xi}$ is a $\mathbb P_{\alpha,\xi}$-name of a 
  ccc poset and $\mathbb P_{\alpha,\xi+1} = \mathbb P_{\alpha,\xi}*
   \dot{\mathbb Q}_{\alpha,\xi}$,
   \item if $\eta$ is a limit, then $\mathbb P_{\beta,\eta} =
    \bigcup\{ \mathbb P_{\beta, \gamma  } : \gamma <\eta\}$.
   
\end{enumerate}
   \end{definition}

One  constructs $\kappa\ttimes\zeta$-iterations by recursion
on $\zeta$ and, for successor steps,
 by careful choice of the
   component sequence
  $\{\dot{\mathbb Q}_{\alpha,\xi} : \alpha \leq\kappa\}$. 
The first important result is that all the work is in the successor
steps.  
 The following is from \cite{BrendleFischer}*{Lemma 3.10}

\begin{lemma} If $\zeta $ is a limit\label{limitMatrix}
ordinal then a family
 $$\langle\langle \mathbb P_{\alpha,\xi} : \alpha\leq\kappa,
 \xi \leq \zeta \rangle,\langle \dot{\mathbb Q}_{\alpha,\xi} : 
 \alpha\leq\kappa,
  \xi <\zeta \rangle \rangle$$     is a 
  $\kappa\ttimes\zeta$-matrix iteration providing
   for all $\eta<\zeta$ and $\beta \leq\kappa$:
   
\begin{enumerate}
\item $\langle\langle \mathbb P_{\alpha,\xi} : \alpha\leq\kappa,
 \xi \leq \eta \rangle,\langle \dot{\mathbb Q}_{\alpha,\xi} : \alpha\leq
 \kappa,
  \xi <\eta \rangle \rangle$  is
  a $\kappa\ttimes\eta$-matrix iteration,  and
  \item $\mathbb P_{\beta,\zeta} = \bigcup\{
   \mathbb P_{\beta,\xi} : \xi<\zeta\}$.
\end{enumerate}
\end{lemma}

The following is well-known, see for example
\cite{Diego13}*{Section 5}
and \cite{SouslinForcing}.

\begin{proposition}
For any  $\zeta $\label{easyStep} 
and $\kappa\ttimes\zeta$-matrix iteration
$$\langle\langle \mathbb P_{\alpha,\xi} : \alpha\leq\kappa,
 \xi \leq \zeta \rangle,\langle \dot{\mathbb Q}_{\alpha,\xi} : \alpha\leq\kappa,
  \xi <\zeta \rangle \rangle$$
  the extension  
 $$\langle\langle \mathbb P_{\alpha,\xi} : \alpha\leq\kappa,
 \xi \leq \zeta{+}1 \rangle,\langle \dot{\mathbb Q}_{\alpha,\xi} : \alpha\leq
 \kappa,
  \xi <\zeta{+}1 \rangle \rangle$$     is a 
  $\kappa\ttimes(\zeta{+}1)$-matrix iteration if either 
  the following holds:
 
\begin{enumerate}

 \item [(1)${}_{\mathbb Q}$]
 for all
  $\alpha\leq\kappa$,
  $\dot {\mathbb Q}_{\alpha,\zeta}$ is the $\mathbb P_{\alpha,\zeta}$-name 
 for $\mathbb D$, 
 \item [(2)${}_{\mathbb Q}$]
 there is an $\alpha<\kappa$ such that $\dot{\mathbb Q}_{\beta,\zeta}$
 is the trivial poset for $\beta <\alpha$, $\dot {\mathbb Q}_{\alpha,\zeta}$
 is a $\mathbb P_{\alpha,\zeta}$-name of a $\sigma$-centered  poset, and
  $\dot {\mathbb Q}_{\beta,\zeta} = \dot {\mathbb Q}_{\alpha,\zeta}$ for all
   $\alpha <\beta\leq \kappa$.
\end{enumerate}
 
\end{proposition}
  
  Notice that if we define the extension as in (1)${}_{
  \mathbb Q}$  then we will be adding a 
  dominating real, but even if $\dot{\mathbb Q}_{\alpha,\zeta}$
  is forced to equal $\mathbb D$ in 
   (2)${}_{
  \mathbb Q}$, the real added will only dominate
  the reals added by $\mathbb P_{\alpha,\zeta}$.

\begin{proposition}\cite{BlassShelah}
 Let $M$ be a    model of (a sufficient amount of) set-theory
 and $P\in M$  be a poset that is also contained in $M$.
  Then for any $f\in \omega^\omega$ that is not dominated
 by any $g\in M\cap \omega^\omega$, 
  $P$ forces that $f\not\leq \dot g$  
  for all\label{nobound}  
   $P$-names $\dot g\in M$ of elements of $\omega^\omega$.
 \end{proposition}

\begin{proof}
 Let $p\in P$ and $n\in\omega$. It suffices to prove
 that there is a $q<p$ in $P$ and a $k>n$ and $m < f(k)$
 such
 that $q\Vdash \dot g(k)=m$. Since $p\in M$, we can work
  in $M$
and  define a function $h\in\omega^\omega$ by the rule
  that, for all $k\in\omega$, there is a $q_k<p$ such
  that $q_k\Vdash \dot g(k) = h(k)$. Choose any $k>n$
  so that $h(k) < f(k)$. Then $q_k\Vdash \dot g(k)<f(k)$
  and proves that $p\not\Vdash f\leq \dot g$.
\end{proof}
  
  An analogous result, with the same proof, holds for splitting.

\begin{proposition}
Let $M$ be a    model of (a sufficient amount of) set-theory
 and $P\in M$  be a poset that is also contained in $M$.
  If   $x\in  [\omega]^\omega$ satisfies\label{nosplit}
that $  y\not\subset x$ for all $y\in M\cap [\omega]^\omega$,
 then $P$ forces that $ \dot y\not\subset x$  
 for all $P$-names $\dot y\in M$ for elements of
  $[\omega]^\omega$.
\end{proposition}

 We also use the main construction from \cite{BlassShelah}.

\begin{proposition}
 Suppose that $$\langle\langle \mathbb P_{\alpha,\xi} : \alpha\leq\kappa,
 \xi \leq \zeta \rangle,\langle \dot{\mathbb Q}_{\alpha,\xi} : \alpha \leq\kappa,
  \xi <\zeta \rangle \rangle$$ is a $\kappa\times\zeta$-matrix iteration
  and that $\{\dot f_\alpha : \alpha<\kappa\}$ is a\label{bsplit}
  sequence satisfying
  that, for all $\alpha<\kappa$
\begin{enumerate}
 \item $\dot f_\alpha$ is a $\mathbb P_{\alpha,\zeta}$-name that is forced to be 
 in $\omega^\omega$,
 \item for all $\beta<\alpha$ and $\mathbb P_{\beta,\zeta}$-name
 $\dot g$ of a member of 
  $\omega^\omega$, $\mathbb P_{\alpha,\zeta}$ forces that $\dot f_\alpha \not< \dot g$.
\end{enumerate}
Then there is a sequence $\{ \dot {\mathcal U}_{\alpha,\zeta}
 : \alpha\leq \kappa\}$ such
that, for all $\alpha<\kappa$:
\begin{enumerate}
\setcounter{enumi}{2}
 \item $\dot{\mathcal U}_{\alpha,\zeta}$ is a $\mathbb P_{\alpha,\zeta}$-name of
 an ultrafilter on $\omega$,
 \item for $\beta<\alpha$, $\dot{\mathcal U}_{\beta,\zeta}$ is a subset
 of $\dot{\mathcal U}_{\alpha,\zeta}$  
 \item for each $\beta<\alpha$ and each $\mathbb P_{\beta,\zeta}*
 \mathbb M(\dot{\mathcal U}_{\beta,\zeta})$-name $\dot g$ of an element of
  $\omega^\omega$, $\mathbb P_{\alpha,\zeta}*
  \mathbb M(\dot{\mathcal U}_{\alpha,\zeta})$ 
  forces that $\dot f_{\alpha}\not< \dot g$, and
  \item $\langle\langle \mathbb P_{\alpha,\xi} : \alpha\leq\kappa,
 \xi \leq \zeta{+}1 \rangle,\langle \dot{\mathbb Q}_{\alpha,\xi} : \alpha\leq\kappa,
  \xi <\zeta{+}1 \rangle \rangle$ is a $\kappa\ttimes(\zeta{+}1)$-matrix iteration, 
  where, for each $\alpha\leq\kappa$, $\mathbb P_{\alpha,\zeta{+}1}
   = \mathbb P_{\alpha,\zeta}*
 \dot {\mathbb Q}_{\alpha,\zeta}$ and  $\dot {\mathbb Q}_{\alpha,\zeta}$ 
 is the $\mathbb P_{\alpha,\zeta}$-name for 
 $\mathbb M(\dot{\mathcal U}_{\alpha,\zeta})$. 
\end{enumerate}
\end{proposition}
 
We record two more well-known preparatory preservation results.

\begin{proposition}[\cite{BaumDordal}]
Suppose that $M\subset N$ are models of
(a sufficient amount of) set-theory\label{Dsplit}
and that $G$ is $\mathbb D$-generic over $N$. If
$x\in N\cap[\omega]^\omega$ does not include
any $y\in M\cap [\omega]^\omega$, it will
not include any $y\in M[G]\cap [\omega]^\omega$.
 \end{proposition}

\begin{proposition}
 Assume that $\{ P_\alpha : \alpha \leq \delta\}$ is a $\cprec$-increasing
 chain of ccc posets with $P_\delta = \bigcup \{ P_\alpha
  :\alpha < \delta\}$.
Let $G_\delta$ be $P_\delta$-generic.
 Let $x\in [\omega]^\omega$ and $f\in \omega^\omega$. Then each
 of the\label{FSpreserve} following hold:
 
\begin{enumerate}
 \item If $f \not\leq g$ for each $g\in V[G_\alpha]$ for all $\alpha<\delta$,
  then $f\not \leq g$ for each $g\in V[G_\delta]$.
  \item If $x$ does not contain any $y\in [\omega]^\omega\cap V[G_\alpha]$
  for all $\alpha<\kappa$, then $x$ does not contain any
   $y\in [\omega]^\omega\cap V[G_\delta]$.
\end{enumerate}
\end{proposition}

\begin{proof} 
We prove only (1) since the proof of (2) is similar.
If $\delta$ has uncountable cofinality, then there
is nothing to prove since $V[G_\delta]\cap \omega^\omega$
would then equal $\bigcup \{ V[G_\alpha]\cap \omega^\omega :
 \alpha < \delta\}$. Otherwise, consider
 any     $P_\delta$-name $\dot g$ 
 and condition 
  $p\in P_\delta $ forcing
  that $\dot g\in \omega^\omega$.
    We prove that $p$ does not force
 that $\dot g(n)>f(n)$ for all $n>k$. We may assume
 that $\dot g$ is a canonical name, 
 so 
let $\dot g  =
 \bigcup\{ (\widecheck{n,m}) \ttimes A_{n,m} : n,m\in\omega\ttimes
 \omega\}$.  
 Choose any $\alpha<\delta$
 so that $p\in P_\alpha$ and work in $V[G_\alpha]$.  We
 define a function $h\in \omega^\omega\cap V[G_\alpha]$. 
 For each $n\in \omega$, we set $h(n)$ to be the minimum
 $m$ such that there is $q_{n,m}\in A_{n,m}$ 
 having a $P_\alpha$-reduct
 $p_{n,m}\in 
 G_\alpha$.  Since $A_n  = \bigcup\{ A_{n,m} : m\in\omega\}$ is predense
 in $P_\kappa$, the set of $P_\alpha$-reducts of members of $A_n$
 is predense in $P_\alpha$. By hypothesis, there is a $k<n$
 such that $h(n)<f(n)$. Since $q_{n,h(n)}$ is compatible with
  $p$, this prove that $p\not\Vdash \dot g(n)>f(n)$.    
\end{proof}

\section{Building the models to
  distinguish $\mathfrak h, \mathfrak b , \mathfrak s$}

For simplicity we assume GCH.
Let $\aleph_1 \leq \mu< \kappa < \lambda$ be regular cardinals and
assume that $\theta>\lambda$ is a cardinal with cofinality $\mu$.
We will 
 need to enumerate
 names in order to force that $\mathfrak p \geq\mu$. For
 each ccc poset $\tilde P\in H(\theta^+)$ let 
 $\{ \dot Y(\tilde P,\xi) : \xi < \theta\}$ be an enumeration of the
 set of all canonical $\tilde P$-names of subsets of $\omega$. Also
 let $\{ S_\xi : \xi < \theta\}$ be an enumeration of
  all subsets of $\theta$ that have cardinality less than $\mu$. For
  each $\eta <\lambda$, let $\zeta_\eta$ denote the ordinal product
   $\theta\cdot  \eta$. 
 
\begin{theorem}
 There\label{main1}
  is a ccc
  poset  that forces $\mathfrak p = \mathfrak h =\mu$, 
  $\mathfrak b =\kappa$, $\mathfrak s = \lambda$ and $\mathfrak c = \theta$.
\end{theorem}

\begin{proof}
The poset  will be obtained by constructing a
 $\kappa\ttimes\zeta$-matrix iteration where $\zeta$ is the
ordinal product $\theta\cdot \lambda
 =\sup\{ \zeta_\eta : \eta < \lambda\}$.
We begin with the $\kappa\ttimes\kappa$-matrix iteration
   $$\langle\langle \mathbb P_{\alpha,\xi} : \alpha\leq\kappa,\xi\leq\kappa\rangle,
   \langle \dot{\mathbb Q}_{\alpha,\xi} : \alpha\leq
   \kappa,\xi<\kappa\rangle
   \rangle$$
   where, for each $\alpha<\kappa$, $\mathbb P_{\alpha,\alpha}$
   forces that $\dot{\mathbb Q}_{\alpha,\alpha}$ is $\mathbb D$,
   for $\beta<\alpha$, $\dot{\mathbb Q}_{\beta,\alpha}$ is the trivial poset,
   and for $\alpha\leq\beta\leq\kappa$, $\dot{\mathbb Q}_{\beta,\alpha}$
   equals $\dot{\mathbb Q}_{\alpha,\alpha}$. By Proposition
   \ref{easyStep}, there is such a matrix. For each $\alpha<\kappa$,
    let $\dot f_\alpha$ be the canonical name for the dominating
    real added by $\mathbb P_{\alpha,\alpha+1}$. 
    By Propositions \ref{nobound}
    and \ref{FSpreserve}, it follows that for all $\beta<\alpha<\kappa$,
     $\mathbb P_{\alpha,\kappa}$ forces that $\dot f_\alpha\not\leq\dot g$
     for all $\mathbb P_{\beta,\kappa}$-names $\dot g$ of elements
     of $\omega^\omega$.
     
   We omit the routine enumeration details involved
   in the recursive construction and state the properties we require
   of our $\kappa\ttimes\zeta$-matrix iteration. 
  Each step of the construction uses either (2) of
  Proposition \ref{easyStep}
  or Proposition \ref{bsplit}
  to choose the next sequence
   $\{\dot{\mathbb Q}_{\alpha,\xi} : \alpha \leq\kappa\}$.
    In the case of Proposition \ref{easyStep} (2),
   the preservation of inductive condition (1) follows from Proposition
    \ref{nobound}. The preservation through limit steps follows
    from Proposition \ref{FSpreserve}.
     
   There is a matrix-iteration
   sequence
   $$\langle\langle \mathbb P_{\alpha,\xi} : \alpha\leq\kappa,\xi\leq\zeta\rangle,
   \langle \dot{\mathbb Q}_{\alpha,\xi} : \alpha\leq\kappa,\xi<\zeta\rangle
   \rangle$$
   satisfying each of the following for each $\xi<\zeta$:
   
\begin{enumerate}
\item for each $\beta<\alpha<\kappa$ and each $\mathbb P_{\beta,\xi}$-name
 $\dot g$ for\label{b} an element of $\omega^\omega$, $\mathbb P_{\alpha,\xi}$
 forces that $\dot f_\alpha\not\leq \dot g$,
 \item for each $\beta<\lambda$ with $\zeta_{\beta+1}\leq\xi$ 
 and each $\eta<\theta$,
 if $\mathbb P_{\kappa,\zeta_\beta}$ forces\label{p} that the family
    $
  \mathcal F_{\beta,\eta} =   \{ \dot Y(\mathbb P_{\kappa,\zeta_\beta},\gamma) : \gamma \in S_\eta\}$    
    has the sfip, then there is a $\bar\eta<\zeta_{\beta+1}$ and an $\alpha<\kappa$
    such that $\dot{\mathbb Q}_{\beta,\bar\eta}$ equals the
     $\mathbb P_{\alpha,\bar\eta}$-name for 
     $\mathbb M(\mathcal F_{\beta,\eta})$
     for all $\alpha\leq\beta\leq\kappa$,
     \item for each $\beta<\lambda$ such that\label{s} $\zeta_\beta<\xi$,
   $\mathbb P_{\kappa,\zeta_\beta+1}$ equals
      $\mathbb P_{\kappa,\zeta_\beta}*\mathbb 
      M(\dot{\mathcal U}_{\kappa,\zeta_\beta})$
      and $\dot{\mathcal U}_{\kappa,\zeta_\beta}$ is a
      $\mathbb P_{\kappa,\zeta_\beta}$-name of an ultrafilter on $\omega$,
      \item for each $\eta<\lambda$ and\label{bb} each $\alpha<\kappa$ such
      that $\zeta_\eta <\xi$, then $\dot {\mathbb Q}_{\alpha,\zeta_\eta+\alpha}$
      is the $\mathbb P_{\alpha,\zeta_\eta+\alpha}$-name for $\mathbb D$,
      and $\dot {\mathbb Q}_{\beta,\zeta_\eta+\alpha} = 
      \dot {\mathbb Q}_{\alpha,\zeta_\eta+\alpha}$ for all $\alpha\leq\beta\leq \kappa$. 
 \end{enumerate}
 Now we verify that $P = \mathbb P_{\kappa,\zeta}$ has the desired properties. 
 Since $P$ is ccc, it preserves cardinals and clearly forces that
  $\mathfrak c = \theta$. It thus follows from Corollary \ref{2.2} that
   $\mathfrak p\leq \mathfrak h\leq \mu  =\operatorname{cf}(\mathfrak c)$. 
 If $\mathcal Y$ is a family of fewer than $\mu$ many canonical $P$-names
 of subsets of $\omega$, then there is an $\alpha<\kappa$ and $\eta<\lambda$
 such that $\mathcal Y$ is a family of $\mathbb P_{\alpha,\zeta_\eta}$-names.
 It follows that there is a $\beta<\theta$ such that $\mathcal Y$ 
 is equal to the set $\{ \dot Y(\mathbb P_{\kappa,\zeta_\beta},\gamma) : 
  \gamma\in S_\eta\}$. If $\mathbb P_{\kappa,\zeta_\beta}$ forces
  that $\mathcal Y$ has the sfip, then  
  inductive condition \ref{p}  ensures that there is a $P$-name 
  for a pseudo-intersection for $\mathcal Y$. This shows that $P$
  forces 
  that $\mathfrak p\geq\mu$.
 It is clear that inductive condition \ref{b} ensures that $\mathfrak b\leq\kappa$. 
 We check that condition \ref{bb} ensure that $\mathfrak b\geq\kappa$. 
 Suppose that $\mathcal G$ is a family of fewer than $\kappa$
 many canonical $P$-names
 of members of $\omega^\omega$. We again find
  $\eta<\lambda$ and $\alpha<\kappa$ such that $\mathcal G$
  is a family of $\mathbb P_{\alpha,\zeta_\eta}$-names. Condition
  \ref{bb} forces there is a function that dominates $\mathcal G$. 
  Finally we verify that condition \ref{s} ensures that $P$ forces
  that $\mathfrak s = \lambda$. If $\mathcal S$ is any family of
  fewer than $\lambda$-many canonical $P$-names of subsets of
   $\omega$, then there is an $\eta<\lambda$ such that
   $\mathcal S$ is a family of $\mathbb P_{\kappa,\zeta_\eta}$-names.
    Evidently, $\mathbb P_{\kappa,\zeta_\eta+1}$ adds a subset of
    $\omega$ that is not split by $\mathcal S$. There are a number
    of ways to observe that for each $\eta<\lambda$, $\mathbb P_{\kappa,
    \zeta_{\eta+1}}$ adds a real that is Cohen over the extension
    by $\mathbb P_{\kappa,\zeta_\eta}$. This ensures that 
    $P$ forces that $\mathfrak s\leq\lambda$.
\end{proof}

In the next result we proceed similarly except that we first
add $\kappa$ many Cohen reals and preserve that they are splitting.
We then cofinally add dominating reals with Hechler's $\mathbb D$
and again use small posets to ensure $\mathfrak p\geq \mu$. 

\begin{theorem}
 There\label{main2}
  is a ccc
  poset  that forces $\mathfrak p = \mathfrak h =\mu$, 
  $\mathfrak s =\kappa$, $\mathfrak b = \lambda$ and $\mathfrak c = \theta$.
\end{theorem}

\begin{proof}
   We begin with the $\kappa\ttimes\kappa$-matrix iteration
 $$\langle\langle \mathbb P_{\alpha,\xi} : \alpha\leq\kappa,\xi\leq\kappa\rangle,
   \langle \dot{\mathbb Q}_{\alpha,\xi} : \alpha\leq\kappa,\xi<\kappa\rangle
   \rangle$$
   where
 $\mathbb P_{\alpha,\alpha}$
   forces that $\dot{\mathbb Q}_{\alpha,\alpha}$ is $\mathcal C_\omega$,
   for $\beta<\alpha$, $\dot{\mathbb Q}_{\beta,\alpha}$ is the trivial poset,
   and for $\alpha\leq\beta\leq\kappa$, $\dot{\mathbb Q}_{\beta,\alpha}$
   equals $\dot{\mathbb Q}_{\alpha,\alpha}$. We let
   $\dot x_\alpha$ denote the canonical Cohen real added by
    $\mathbb P_{\alpha,\alpha+1}$. Of course
     $\mathbb P_{\alpha,\alpha+1}$ forces that neither
      $\dot x_\alpha$ nor its complement include any infinite
      subsets of $\omega$ that have, for any $\beta<\alpha$,
       a $\mathbb P_{\beta,\alpha+1}$-name. 
      By Proposition \ref{FSpreserve}, the inductive
      condition \ref{s2} below holds for $\xi =\kappa$.

Then, proceeding as in the proof of Theorem \ref{main1},  
  we   just assert the existence of a $\kappa\ttimes\zeta$-matrix
   iteration    
   $$\langle\langle \mathbb P_{\alpha,\xi} : \alpha\leq\kappa,\xi\leq\zeta\rangle,
   \langle \dot{\mathbb Q}_{\alpha,\xi} : \alpha\leq\kappa,\xi<\zeta\rangle
   \rangle$$
   satisfying each of the following for each $\kappa\leq \xi<\zeta$:

\begin{enumerate}
 \item for each $\beta < \alpha<\kappa$, $\mathbb P_{\alpha,\xi}$
 forces that neither $\dot x_\alpha$ nor $\omega\setminus \dot x_\alpha$
 contains any infinite subset of $\omega$ that has a
  $\mathbb P_{\beta,\xi}$-name\label{s2},
 
 \item for each $\eta<\lambda$ with $\zeta_{\eta+1}\leq\xi$ 
 and each $\delta<\theta$,
 if $\mathbb P_{\kappa,\zeta_\eta}$ forces\label{p2} that the family
    $
  \mathcal F_{\eta,\delta} =   
  \{ \dot Y(\mathbb P_{\kappa,\zeta_\eta},\gamma) : \gamma \in S_\delta\}$    
    has the sfip, then there is a $\bar\delta<\zeta_{\eta+1}$ 
    and an $\alpha<\kappa$
    such that $\dot{\mathbb Q}_{\beta,\bar\delta}$ equals the
     $\mathbb P_{\alpha,\bar\delta}$-name for 
     $\mathbb M(\mathcal F_{\eta,\delta})$
     for all $\alpha\leq\beta\leq\kappa$,

      \item for each $\eta<\lambda$ and\label{ss2} each $\alpha<\kappa$ such
      that $\zeta_\eta <\xi$, then $\dot {\mathbb Q}_{\alpha,\zeta_\eta+\alpha}$
      is the $\mathbb P_{\alpha,\zeta_\eta+\alpha}$-name
      for      $\mathbb M(\dot{\mathcal U}_{\alpha,\zeta_\beta})$
           where 
       $\dot{\mathcal U}_{\alpha,\zeta_\beta}$ is a
      $\mathbb P_{\alpha,\zeta_\beta}$-name of an ultrafilter on $\omega$,
      and $\dot {\mathbb Q}_{\beta,\zeta_\eta+\alpha} = 
      \dot {\mathbb Q}_{\alpha,\zeta_\eta+\alpha}$ for all $\alpha\leq\beta\leq \kappa$. 

  \item for each $\eta<\lambda$ such that\label{b2} $\zeta_\eta<\xi$,
   $\mathbb P_{\kappa,\zeta_\eta+1}$ equals
      $\mathbb P_{\kappa,\zeta_\eta}*\mathbb D$,

\end{enumerate}
  Evidently conditions (\ref{p2}) and (\ref{ss2}) are similar and
  can be achieved while preserving condition (\ref{s2}) by
  Proposition \ref{easyStep} (2). The fact that $\mathbb P_{\kappa,
  \zeta_\eta}*\mathbb D$ preserves condition (\ref{s2}) follows
  from Proposition \ref{Dsplit}. Condition (\ref{s2}) ensures that
   $\mathfrak s\leq\kappa$, and by arguments similar to those
   in Theorem \ref{main1}, condition (\ref{ss2}) ensures that
    $\mathfrak s\geq \kappa$. The fact that $\mathfrak b=\lambda$
    (in fact $\mathfrak d =\lambda$) follows easily from 
    condition (\ref{b2}). 
 The facts that    that $\mathfrak c=\theta$,
  $\mathfrak p\geq \mu$ and $\mathfrak h =\mu$
 are proven exactly  as in Theorem \ref{main1}. 
\end{proof}

\section{On $(\kappa,\lambda)$-shattering}

In this section we prove, see Theorem 
\ref{noth}, that it is consistent
that strongly $(\kappa,\kappa^+)$-shattering families
exist.   We will use the method of matrix of posets
from Definition \ref{matrixposet} in which our
  main component   posets to raise
  the value of $\mathfrak s$ will be the Laver
style posets.
 We recall some notions and results
about these studied in  
\cites{Shelah2011,ssingular,Shelahpseudo}. Before proceeding 
we summarize
the rough idea of how we generalize the fundamental preservation technique
of a matrix iteration. In a $\kappa\times\kappa^+$-matrix iteration, one
may introduce a sequence
 $\{\dot a_\alpha : \alpha <\kappa\}$ of $P_{\kappa,1}$-names that
  have no infinite pseudointersection. With this fixed enumeration, one
  then recursively 
 ensures that, for $\gamma<\kappa^+$,
 no $P_{\alpha,\gamma}$-name will be a subset
  of $\dot a_\beta$ for any $\beta\geq\alpha$.
   In the construction introduced in \cite{Shelahpseudo},
   we instead continually add to the list a $P_{0,\gamma{+}1}$-name
    $\dot a_\gamma$ and at stage $\mu<\kappa^+$, we adopt
    a new enumeration of $\{ \dot a_\alpha : \alpha <\mu\}$ in
    order-type $\kappa$ (coherent with previous   enumerations) and again
    ensure that no $P_{\alpha,\mu+1}$-name is a subset of 
    any $\dot a_\beta$ for $\beta$ not listed before $\alpha$
    in this new $\mu$-th enumeration.
  We utilize a $\square$-principle to make
    these enumerations sufficiently coherent. The greater flexibility in
    the definition of $\kappa\times\kappa^+$-matrix of posets makes
    this possible.

\begin{proposition}[\cite{Shelah2011}*{1.9}] If $P\cprec P'$ are ccc posets,
and $\dot{\mathcal D}\subset \dot{\mathcal E}$ 
 are, respectively, a $P$-name and a $P'$-name, of 
 ultrafilters on $\omega$, then 
$P*\mathbb L(\dot{\mathcal D})\cprec P'*\mathbb L(\dot{\mathcal E})$.
\end{proposition}

\begin{definition}
 A family $\mathcal A\subset[\omega]^\omega$ is thin over a
 model $M$ if for every $I$ in the ideal generated by $\mathcal A$
 and every  infinite family $\mathcal F\in M$ consisting of
 pairwise disjoint finite sets of bounded size, 
 $I$ is disjoint from some member of $\mathcal F$.
\end{definition}

   It is routine to prove that, for each limit ordinal $\delta$,
    $\mathcal C_\delta$ forces
   that the family $\{ \dot x_\alpha : \alpha\in  \delta\}$,
   as defined above,
   is thin over the ground model. In fact if $\mathcal A$ is 
   thin over some model $M$, then $\mathcal C_\delta$
   forces that $\mathcal A\cup \{\dot x_\alpha :\alpha\in \delta\}$
   is also thin over $M$.    This is the notion we use
   to control that property (1) of the definition of a
    $(\kappa,\kappa^+)$-shattering sequence 
   will be preserved while at the same time raising
   the value of $\mathfrak s$.

We first note that Proposition  \ref{nosplit} extends
to include this concept.

\begin{proposition}
Suppose that   $M $ is a model of
(a sufficient amount of) set-theory\label{nullstep}
and that $\mathcal A\subset[\omega]^\omega$ is
thin over $M$. Then for any poset $P$ such
that $P\in M$ and $P\subset M$, $\mathcal A$
is thin over the forcing extension by $P$.
   \end{proposition}

\begin{proof}
 Let $\{ \dot F_\ell : \ell\in\omega\}$ be $P$-names and
 suppose that $p\in P$ forces that $\{\dot F_\ell : \ell\in\omega\}$
 are pairwise disjoint subsets of $  [\omega]^k$. Also let
 $I$ be any member of the ideal generated by $\mathcal A$.
 Working in $M$, 
recursively choose $q_j < p $ ($j\in\omega$) and $H_j,\ell_j$ 
 so that $q_j \Vdash \dot F_{\ell_j} = \check{H}_j$ 
 and $H_j\cap \bigcup\{ H_i : i<j\} =\emptyset$.  The
 sequence $\{ H_j : j\in \omega\}$ is a family in $M$
 of pairwise disjoint sets of cardinality $k$. Therefore
 there is a $j$ with $H_j\cap I=\emptyset$. This proves
 that $p$ does not force that $I$ meets every
 member of $\{\dot F_\ell :\ell\in\omega\}$. 
\end{proof}

\begin{lemma}[\cite{Shelahpseudo}*{3.8}]
Let $\kappa$ be a regular uncountable cardinal
and let
 $\{P_\beta :\beta\leq\kappa\}$
   be a $\cprec$-increasing chain of
    ccc posets  with $P_\kappa = \bigcup\{P_\alpha
     :\alpha < \kappa\}$. 
    Assume that, for each $\beta<\kappa$,\label{usingL}
$\dot{\mathcal A}_\beta$ is a $P_{\beta+1}$-name of
a subset of $ [\omega]^\omega$ that is forced to be thin over
the forcing extension by $P_\beta$.
 Also let 
   $\dot{\mathcal D}_0 $ be a $P_0*
   \mathcal C_{\{0\}\ttimes\mathfrak c}$-name that is forced
   to be 
     a Ramsey ultrafilter on $\omega$. 
     Then there is a sequence
      $\langle \dot{\mathcal D}_\beta : 0<\beta <\kappa\rangle$
      such that for all $\alpha < \beta <\kappa$:
\begin{enumerate}
\item $\dot {\mathcal D}_\beta$ is a $P_\beta*
\mathcal C_{(\beta{+}1)\ttimes\mathfrak c}$-name,
\item $\dot{\mathcal D}_\alpha$ is a subset of $\dot{\mathcal D}_\beta$,
 \item $P_\beta *\mathcal C_{(\beta{+}1)\ttimes
 \mathfrak c}$ forces that $\dot{\mathcal D}_\beta$
  is a Ramsey ultrafilter, 
  \item $P_\alpha *\mathcal C_{(\alpha{+}1)\ttimes\mathfrak c}
  *\mathbb L(\dot{\mathcal D}_\alpha)\cprec
   P_\beta*\mathcal C_{(\alpha{+}1)\ttimes\mathfrak c}
   * \mathbb L(\dot{\mathcal D}_\beta)$,    
  and
 \item  $P_\beta*\mathcal C_{(\beta{+}1)\ttimes
 \mathfrak c}*\mathbb L(\dot{\mathcal D}_\beta)$
 forces that $\dot{\mathcal A}_\beta$ is thin over the forcing extension
 by $P_\alpha*\mathcal C_{(\alpha{+}1) \ttimes\mathfrak c}*
 \mathbb L(\dot{\mathcal D}_\alpha)$.
\end{enumerate}
 \end{lemma}

\begin{lemma}[\cite{Shelahpseudo}*{2.7}]
Assume that $P_{0,0}\cprec P_{1,0}$ and 
that $\dot{\mathcal A}$ is a  $P_{1,0}$-name of a subset
of $[\omega]^\omega$.  Assume\label{Lchains}
 that $\langle
 P_{0,\xi} : \xi <\delta\rangle$ and $\langle P_{1,\xi} : \xi<\delta\rangle$
  are $\cprec$-chains such that $P_{0,\xi}\cprec P_{1,\xi}$ for all 
   $\xi<\delta$, and that $P_{1,\xi}$ forces that
    $\dot{\mathcal A}$ is thin over the forcing extension by
     $P_{0,\xi}$ for all $\xi< \delta$. Then $P_{1,\delta}=
     \bigcup\{P_{1,\xi} : \xi<\delta\}$ forces that 
      $\mathcal A$ is thin over the forcing extension by
       $P_{0,\delta}=\bigcup\{P_{0,\xi} : \xi <\delta\}$.
\end{lemma}

Before proving this next result we recall the notion of
a $\square_\kappa$-sequence. For a set $C$ of ordinals,
let $\sup(C)$ be the supremum, $\bigcup C$, of $C$ and let
$\acc(C)$ denote the set of limit ordinals $\alpha<\sup(C)$ such that
$C\cap \alpha$ is cofinal in $\alpha$. For a limit ordinal
$\alpha$, a  set $C$  is a
 cub in   $\alpha$ if $C\subset\alpha=\sup(C)$ 
 and $\acc(C)\subset C$.

\begin{definition}[\cite{Jech2}]
For a cardinal 
 $\kappa$, the family $\{ C_\alpha : \alpha \in 
\acc(\kappa^+)\}$
 is a $\square_\kappa$-sequence if, for each $\alpha\in
\acc(\kappa^+)$:
 
\begin{enumerate}
 \item $C_\alpha$ is a cub in $\alpha$,
 \item if $\operatorname{cf}(\alpha)<\kappa$, then $|C_\alpha|<\kappa$,
 \item if $\beta\in \acc(C_\alpha)$, then $C_\beta = C_\alpha\cap \beta$.
\end{enumerate}
If there is a 
$\square_\kappa$-sequence, then $\square_\kappa$ is said to hold. 
\end{definition}

\begin{theorem}  It is consistent\label{noth} with $\aleph_1 < 
\mathfrak h<\mathfrak s <\operatorname{cf}(\mathfrak c )=\mathfrak c $
that there is a 
 $(\mathfrak h,\mathfrak s )$-shattering family. 
\end{theorem}

\begin{proof}
 We start in a model of GCH satisfying $\square_\kappa$ 
 for some   regular cardinal
  $\kappa>\aleph_1$.  Choose any regular $\lambda>\kappa^+$.  
  Fix a $\square_\kappa$-sequence
   $\{ C_\alpha : \alpha \in \acc
(\kappa^+)\}$. We may
   assume that $C_\alpha = \alpha$ for all $\alpha\in
\acc(\kappa)$.
   For each $\alpha\in \acc(\kappa^+)$, let $o(C_\alpha)$ denote
   the order-type of $C_\alpha$.
    When 
$\acc(C_\alpha)$ is bounded in $\alpha$ with $\eta=\max(\acc(C_\alpha))$,
    then let $\{\varphi^\alpha_\ell:\ell\in\omega\}$ enumerate
  $C_\alpha\setminus\eta $ in increasing order.
  
   We will construct
   a $\kappa\ttimes\kappa^+$-matrix of posets, 
    $\langle P_{\alpha,\xi} : \alpha<\kappa,\xi<\kappa^+ \rangle
    \in H(\lambda^+)$
    and prove that the 
    poset $P_{\kappa,\kappa^+}$ as in 
    Lemma \ref{limit} has the desired properties.  
    For
    each $\eta<\xi<\kappa^+$, we will also
    choose an     
    $\iota({\eta,\xi})\in \kappa$ 
  satisfying, as in (3)
  of the definition of $\kappa\ttimes(\xi{+}1)$-matrix,
   that $P_{\alpha,\eta}\cprec P_{\alpha,\xi}$
    for all $\iota({\eta,\xi})\leq \alpha <\kappa$.
    We construct this family by recursion on $\xi<\kappa^+$,
    and, 
   for each $\xi<\kappa^+$, we let $P_{\kappa,\xi}$ 
   denote the poset $\bigcup\{ P_{\alpha,\xi} : \alpha < \kappa\}$
   as in Lemma \ref{limit}. 
   
   We will recursively define two other families.
     For each $\alpha<\kappa$ and
     $\xi<\kappa^+$, we will define
      a set $\operatorname{supp}(P_{\alpha,\xi})\subset\xi$ that can
     be viewed as the union of the supports of the elements
     of $P_{\alpha,\xi}$ and will satisfy that
      $ \{\operatorname{supp}(P_{\alpha , \xi}) : \alpha <\kappa\}$ is increasing
      and covers $\xi$.
     For each  limit $\eta<\kappa^+$ of cofinality less than $\kappa$
     and each $n\in\omega$,
      we will select a canonical $P_{\kappa,\eta{+}n{+}1}$-name,
       $\dot a_{\eta{+}n} $ 
       of 
       a subset $\omega$ that is forced to be Cohen over
       the forcing extension by $P_{\kappa,\eta}$.  While
       this condition looks awkward, we simply want to avoid
       this task at limits of cofinality $\kappa$.
       Needing notation for this, let $E=  
        \kappa^+\setminus \bigcup\{ [\eta,\eta+\omega) 
        : \operatorname{cf}(\eta)=\kappa\}$.
       
       For each $\alpha<\kappa$
       and $\xi<\eta<\kappa^+$, 
        we define   $\mathcal A_{\alpha,\xi,\eta}$
        to be the family $\{ \dot a_{\gamma} : \gamma
        \in E\cap \eta\setminus
        \operatorname{supp}(P_{\alpha,\xi})\}$. The 
        intention is that for all $\alpha<\kappa$
        and $\xi\leq \eta<\kappa^+$, 
        $\mathcal A_{\alpha,\xi,\eta}$ is a family of
       $P_{\kappa,\eta}$-names which is forced by
        the poset
         $P_{\kappa,\eta}$ to be
         thin over the forcing extension by $P_{\alpha,\xi}$. 
         Let us note that if $\alpha<\beta$ and $\xi\leq\eta<\kappa^+$,
         then $\mathcal A_{\alpha,\xi,\eta}\setminus
          \mathcal A_{\beta,\eta,\eta}$ should then be  a set of
           $P_{\beta,\eta}$-names.
By ensuring that $\operatorname{supp}(P_{\alpha,\xi})$ has cardinality
less than $\kappa$ for all $\alpha<\kappa$ and $\xi<\kappa^+$,
 this will ensure that the family
  $\{ \dot a_{\eta} : \eta \in E\}$ is $(\kappa,\kappa^+)$-shattering. 
  For each $\eta<\kappa^+$ with cofinality $\kappa$ we will
  ensure that $P_{\kappa,\eta+1}$ has the form 
   $P_{\kappa,\eta}*\mathcal C_{\kappa\ttimes\lambda}$
   and that $P_{\kappa,\eta+2} = P_{\kappa,\eta+1}*
   \mathbb L(\dot{\mathcal D}_{\kappa,\eta} )$ for
   a $P_{\kappa,\eta+1}$-name $\dot{\mathcal D}_{\kappa,\eta}$ of
   an ultrafilter on $\omega$. This will ensure that $\mathfrak c\geq\lambda$
   and
   $\mathfrak s=\kappa^+$. 
   The sequence defining 
   $P_{\kappa,\eta+3}$ will be devoted to ensuring
   that $\mathfrak p\geq \kappa$.
   
     We start the recursion in a rather trivial fashion.
    For each $\alpha< \kappa$, 
    $P_{\alpha,0} = \mathcal C_\omega$ and, for each $n\in\omega$,
 $P_{\alpha,n+1} = P_{\alpha,n}*\mathcal C_\omega$.  
      We may also let $\iota({n,m})= 0$ for all $n<m<\omega$. 
For each $n\in\omega$,  let $\dot a_n$ be the canonical
name of the Cohen real added by the second coordinate
of 
$P_{\kappa,n+1} = P_{\kappa,n}*\mathcal C_\omega$.
         For each $\alpha<\kappa$ and $n\in \omega$, define
      $\operatorname{supp}(P_{\alpha,n})$ to be $n$. 
It should be clear that $P_{\kappa,\omega}$ forces
that, for each $\alpha<\kappa$ and $n\in\omega$,
the family $\{\dot a_m : n\leq m\in\omega\}$ is thin
over the forcing extension by $P_{\alpha,n}$.     
Assume that $P$ is a poset whose elements are functions
with domain a subset of an ordinal $\xi$.  
 We adopt the notational convention
 that   for a $P$-name $\dot Q$ for  a poset, 
 $P*_\xi \dot Q$  will denote the representation
 of $P*\dot Q$ whose elements have
 the form $p\cup \{(\xi,\dot q)\}$ for $(p,\dot q)\in P*\dot Q$.

   We will prove,
   by induction on limit $\zeta<\kappa^+$,
   there is a $\kappa\ttimes \zeta$-matrix 
    $\{ P_{\alpha,\xi} : \alpha<\kappa, \xi < \zeta\}$ satisfying 
    conditions (1)-(10):
\begin{enumerate}
\item for all $\alpha < \beta <  \kappa$ and $\xi<\eta <\zeta $,
if $P_{\alpha,\xi}\cprec P_{\beta,\eta}$, then
the poset $P_{ \beta,\eta}$ forces that the 
family $\mathcal A_{\alpha,\xi ,\eta}\setminus
\mathcal A_{\beta,\eta,\eta}$ is thin over
the forcing extension by $P_{\alpha,\xi}$,
\item for all $\alpha<\kappa$ and $\xi < \zeta$, the elements
$p$ of the poset $P_{\alpha,\xi}$ are functions that have
 a finite domain, $\dom(p)$, contained in $\xi$, 
 \item if  $\acc(C_\zeta)$
 is cub in $\zeta$ and $\eta\in \acc(C_\zeta)$,
  then 
\begin{enumerate}
\item $P_{n,\zeta}$ is the trivial poset and
  $\operatorname{supp}(P_{n,\zeta})=\emptyset$   
for $n\in\omega$,
\item $P_{\alpha,\zeta} = P_{\alpha,\eta}$ and $\operatorname{supp}(P_{\alpha,\zeta})
= \operatorname{supp}(P_{\alpha,\eta})$ 
for all $o(C_\eta)\leq
  \alpha < o(C_\eta)+\omega$, and 
  \item $P_{\alpha,\zeta} =
   \bigcup\{ P_{\alpha,\eta} : \eta\in 
\acc(C_\zeta)\}$ and $\operatorname{supp}(
   P_{\alpha,\zeta}) = \bigcup\{\operatorname{supp}(P_{\alpha,\eta}) :
   \eta\in \acc(C_\zeta)\}$, 
   for all
    $o(C_\zeta)\leq \alpha < \kappa$,
\end{enumerate}
also, let $\iota(\eta,\zeta)=o(C_\eta)$ for all $\eta\in 
\acc(C_\zeta)$
and, 
for all $\gamma<\zeta\setminus \acc(C_\zeta)$,  
let $\iota(\gamma,\zeta)
=  \iota
(\gamma, \eta)$ where $\eta=\min(\acc(C_\zeta)\setminus \gamma)$,
 
  \item if $\max(\acc(C_\zeta)){<}\zeta$ then
  let\\
   \centerline{$\iota_\zeta=\max(
 o(C_\zeta),  \sup\{\iota(\varphi^\zeta_{\ell},\varphi^\zeta_{\ell'}+n)
  : \ell
  \leq\ell' < n<\omega\})$ {and}}  
  
\begin{enumerate}
\item  
  set $P_{\alpha,\zeta} = P_{\alpha,\varphi^\zeta_0 }$
  and $\operatorname{supp}(P_{\alpha,\zeta})=
  \operatorname{supp}(P_{\alpha,\varphi^\zeta_0})$
   for 
  all $\alpha <\iota_\zeta$,
  \item set, for 
  $\iota_\zeta\leq\alpha<\kappa$, 
    $P_{\alpha,\zeta} = \bigcup\{
  P_{\alpha,\varphi^\zeta_\ell+n} :  \ell,n\in\omega\}$
  and 
 $\operatorname{supp}( P_{\alpha,\zeta}) = \bigcup\{
  \operatorname{supp}(P_{\alpha,\varphi^\zeta_\ell+n}) :  \ell,n\in\omega\}$
   
  \item for each $\gamma\in \varphi^\zeta_0$
    let $\iota(\gamma,\zeta) = 
    \iota(\gamma,\varphi^\zeta_0)$, 
     let $\iota(\varphi^\zeta_0,\zeta) = o(C_\gamma)$, and
  for each $\varphi^\zeta_0<\gamma<\zeta$,     
  $\iota(\gamma,\zeta) $ is the maximum of $   \iota_\zeta$
  and $\min\{ 
  \iota(\gamma,\varphi^\zeta_\ell{+}n)  : \ell,n\in\omega
  \ \mbox{and}\  \gamma<
  \varphi^\zeta_\ell{+}n \}$
 
\end{enumerate}
  
  \item if $o(C_\zeta)<\kappa$, then for all $\alpha<\kappa$
  and $n\in\omega$\label{aalpha}
  \begin{enumerate}
  \item $P_{\alpha,\zeta{+}n{+}1} = 
  P_{\alpha,\zeta{+}n}*_{\zeta+n}\mathcal C_\omega$,
  \item  $\dot a_{\zeta{+}n}$ in the canonical $P_{0,\zeta{+}n}
  *_{\zeta+n}\mathcal C_\omega$-name for
   the Cohen real added by the second coordinate copy of
   $\mathcal C_\omega$,
  \item $\operatorname{supp}(P_{\alpha,\zeta{+}n{+}1}) =
  \operatorname{supp}
  (P_{\alpha,\zeta})\cup [\zeta,\zeta{+}n]$,
  and 
  \item $\iota(\zeta{+}k,\zeta{+}n{+}1) = 0$ for all $k\leq n$,
  and,
  for all $\gamma<\zeta$, $\iota(\gamma,
  \zeta{+}{n}{+}1) = \iota(\gamma,\zeta)$,
  \end{enumerate}
  \item if $o(C_\zeta)=\kappa$, then for all 
   $\alpha<\kappa$, $P_{\alpha,\zeta+1} = P_{\alpha,\zeta} 
    *_{\zeta} \mathcal C_{\alpha+1\ttimes\lambda}$, 
    \item if $o(C_\zeta)=\kappa$, then for all $n\in\omega$
    and all $\alpha<\kappa$, $P_{\alpha,\zeta{+}3{+}n} = P_{\alpha,\zeta{+}3}$,
    \item if $o(C_\zeta)=\kappa$, then 
     there is an $\iota_\zeta<\kappa$ such
     that 
     $P_{\beta,\zeta+2} = P_{\beta,\zeta+1}$ for all $\beta <\iota_\zeta$,
    and there is\label{zeta2}
 a sequence $\langle \dot{\mathcal D}_{\alpha, \zeta} : \iota_\zeta\leq
 \alpha < \kappa\rangle$ such that, for each $\iota_\zeta\leq
 \alpha<\kappa$:
     
\begin{enumerate}  
 \item   $\dot{\mathcal D}_{\alpha,\zeta}$
 is a $P_{\alpha,\kappa+1}$-name of a Ramsey ultrafilter on $\omega$,
 \item for each $ \iota_\zeta\leq \beta < \alpha $,
  $\dot{\mathcal D}_{\beta,\zeta}\subset \dot{\mathcal D}_{\alpha,\zeta}$,

  \item  
     $P_{\alpha,\zeta+2} = P_{\alpha,\zeta+1}*_{\zeta{+}1}\mathbb L(
     \dot{\mathcal D}_{\alpha,\kappa})$,
\end{enumerate}
 \item if $o(C_\zeta)=\kappa$, then for   $\iota_\zeta$ chosen as in 
 (\ref{zeta2})
 
\begin{enumerate}
 \item for each $\alpha < \iota_\zeta$, 
 $P_{\alpha,\kappa+3} = P_{\alpha,\kappa+2}$, 
 \item $P_{\iota_\zeta,\zeta{+}3}  = 
 P_{\iota_\zeta, \zeta{+}2}*_{\zeta{+}2}\dot Q_{\iota_\zeta,\zeta{+}2}$
for some  $P_{\iota_\zeta,\zeta}$-name, $\dot Q_{\iota_\zeta,\zeta{+}2}$
in $ H(\lambda^+)$
  of a finite support
 product of $\sigma$-centered posets,
  \item for each $\iota_\zeta<\alpha<\kappa$, 
      $P_{\alpha,\zeta+3} = P_{\alpha,\zeta+2}*_{\zeta{+}2}
      \dot Q_{\iota_\zeta,\zeta{+}2}$,
 \end{enumerate}
\item if $o(C_\zeta)=\kappa$, then for all $\alpha<\kappa$,
$n\in\omega$, and $\gamma<\zeta$,\\
   $\operatorname{supp}(P_{\alpha,\zeta{+}n{+}1}) = \operatorname{supp}(
   P_{\alpha,\zeta})\cup [\zeta,\zeta+n]$, 
    $\iota(\gamma,\zeta{+}n)=\iota(\gamma,\zeta)$, and
     $\iota(\zeta{+}k,\zeta{+}n) = \iota_\zeta$ for all $k<n\in\omega$,
\end{enumerate}
 
    It should be clear from the properties, and by induction
    on $\zeta$,
    that  for 
 all $\alpha<\kappa$ and  $\xi<\zeta$,
 each 
   $p\in P_{\alpha,\xi}$  is a function 
  with finite domain contained in $\operatorname{supp}(P_{\alpha,\xi})$. 
Similarly, it is immediate from the  hypotheses
that $\operatorname{supp}(P_{\alpha,\xi})$ has cardinality less
than $\kappa$ for all $(\alpha,\xi)\in \kappa\ttimes\kappa^+$.

 Before verifying the construction, we first prove,
by induction on $\zeta$,
  that, the conditions (2)-(10) ensure that 
  for all $ \xi\leq \zeta$ and $\eta\in 
\acc(C_\xi)$,

\medskip

 \noindent Claim (a):
  $P_{\alpha,\eta}\cprec P_{\alpha,\xi}$ for all $o(C_\eta)+\omega
  \leq 
 \alpha\in \kappa$,
 
 \noindent Claim (b): $P_{\alpha,\eta} = P_{\alpha,\xi}$ for all 
  $\alpha < o(C_\eta)+
  \omega$
 \medskip
 
\noindent If $o(C_\xi)\leq \alpha$, then  
 $P_{\alpha,\eta}\cprec P_{\alpha,\xi}$
 follows immediately from
 clause 2(c) and, by induction, clauses 3(a). 
   Now
 assume $\alpha <o(C_\xi)+\omega$.
 If $\acc(C_\xi)$ is not cofinal in $\xi$, then, by induction,
  $P_{\alpha,\eta} = P_{\alpha,\varphi^\xi_0}$
  and by clause 3(a) $P_{\alpha,\varphi^\xi_0} = P_{\alpha,\xi}$. 
  If $\acc(C_\xi)$ is cofinal in $\xi$, then choose $\bar\eta\in 
\acc(C_\xi)$
  so that $o(C_{\bar\eta})\leq \alpha < o(C_{\bar\eta})+\omega$.
  By clause 2(b), $P_{\alpha,\xi} = P_{\alpha,\bar\eta}$. By
  the inductive assumption, $P_{\alpha,\eta} = P_{\alpha,\bar\eta}$
  since one of $\eta=\bar\eta, \eta\in 
\acc(C_{\bar\eta}) $ or $\bar\eta\in
  \acc(C_{\eta})$ must hold.

        The second thing we check is that the conditions 
        (2)-(10) also ensure
     that, for each $\zeta <\kappa^+$,
      $\langle P_{\alpha,\eta} : \alpha <\kappa, \eta<\zeta\rangle$
     is a $\kappa\ttimes\zeta$-matrix. 
  We assume, by induction on limit
  $\zeta$, that for $\gamma<\eta<\zeta$,
  $\{ P_{\alpha,\gamma} : \alpha<\kappa\}$ is a $\cprec$-chain
  and that 
   $P_{\alpha,\gamma}\cprec P_{\alpha,\eta}$ for all $\eta$
 with $\iota(\gamma,\eta)\leq \alpha < \kappa$.       
 We check the details for $\zeta+1$ and skip the easy subsequent
 verification for $\zeta+n$ ($n\in\omega)$.    
 Suppose first that 
$\acc(C_\zeta)$ is cofinal in $\zeta$ and let
 $\iota(\gamma,\zeta) \leq \alpha < \kappa$ for some $\gamma<\zeta$.
 Of course we may assume that $\gamma\notin 
\acc(C_\zeta)$. Since
 $\acc(C_\zeta)$
 is cofinal in $\zeta$, let $\eta = \min(\acc(C_\zeta)\setminus \gamma)$. 
 By induction, $P_{\alpha,\gamma}\cprec P_{\alpha,\eta}\cprec P_{\alpha,\zeta}$. 
Now assume that 
$\acc(C_\zeta)$ is not cofinal in $\zeta$. 
If $\gamma\leq \varphi^\zeta_0$, then $\iota(\gamma,\zeta)=
\iota(\gamma,\varphi^\zeta_0)$, and so we have
that $P_{\alpha,\gamma}\cprec P_{\alpha,\varphi^\zeta_0}\cprec
 P_{\alpha,\zeta}$. If $\varphi^\zeta_0<\gamma$, then choose
 any $\ell\in\omega$ so that $\gamma<\varphi^\zeta_\ell$. 
By construction, $\iota(\gamma,\zeta) \geq
 \iota(\gamma,\varphi^\zeta_\ell)$ and so, for
  $\iota(\gamma,\zeta)\leq \alpha<\kappa$,
   $P_{\alpha,\gamma}\cprec P_{\alpha,\varphi^\zeta_\ell} \cprec
    P_{\alpha,\zeta}$.

     Now we consider the values of $\mathcal A_{\alpha,\xi,\eta}$
     for $\alpha<\kappa$ and $\omega\leq\xi\leq\eta$
     by examining the names $\dot a_{\gamma}$ for 
  $\gamma\in E $.

     By clause (\ref{aalpha}), $\dot a_{\gamma}$ 
     is a $P_{0,\gamma{+}1}$-name and 
      $\gamma$ is in the domain of each $p
      \in P_{0,\gamma{+}1}$ appearing
      in the name. 
One direction of this next claim is then obvious  given
     that the domain of every element of $P_{\alpha,\xi}$
     is a subset of $\operatorname{supp}(P_{\alpha,\xi})$.
      
     \smallskip

 \noindent  Claim (c):    $\dot a_\gamma$ is a $P_{\alpha,\xi}$-name,
  if and only if
  $\gamma\in\operatorname{supp}(P_{\alpha,\xi})$.
 \smallskip
 
 Assume that $\gamma\in \operatorname{supp}(P_{\alpha,\xi})$. 
 We prove this by induction on $\xi$.
 If $\xi$ is a limit, then $\operatorname{supp}(P_{\alpha,\xi})$ 
 is defined as a union, hence there is an $\eta <\xi$ such
 that $\gamma\in \operatorname{supp}(P_{\alpha,\eta})$
 and $P_{\alpha,\eta}\cprec P_{\alpha,\xi}$.  If $\xi = \eta+n$
 for some limit $\eta$ and $n\in\omega$, then  
  $P_{\alpha,\eta}\cprec P_{\alpha,\xi}$ and so we may
  assume that $\eta \leq \gamma =\eta+k < \eta+n$ 
  and that $o(C_\eta)<\kappa$. Since $
  P_{0,\eta{+}k}\cprec P_{\alpha,\eta{+}k}
  \cprec P_{\alpha,\eta{+}n}=P_{\alpha,\xi}$, it follows that
  $\dot a_{\gamma}$ is a $P_{\alpha,\xi}$-name.
   \medskip
     
     We prove by induction on $\xi$ ($\xi$ a limit) that 
     for all $\gamma<\xi$:
     \medskip

\noindent  Claim (d):      for all  $\alpha < \iota(\gamma{+}1,\xi)$,
      $\gamma$ is 
       not  in  $\operatorname{supp}(P_{\alpha,\xi})$.
 \medskip

\noindent     First consider
the case      that  $\acc(
C_\xi)$ is cofinal in $\xi$ and let 
$\eta$ be the minimum element of $\acc(C_\xi)\setminus(\gamma{+}1)$.
 By definition
 $\iota(\gamma{+}1,\xi)$ is equal to $\iota(\gamma{+}1,\eta)$
 and the claim follows 
 since we have that 
 $\operatorname{supp}(P_{\iota(\gamma{+}1,\xi),\zeta}) = 
 \operatorname{supp}(P_{\iota(\gamma{+}1,\xi),\eta})$.
 Now assume that 
$\acc(C_\xi)$ is not cofinal in $\xi$
 and assume that $\alpha <\iota(\gamma{+}1,\xi)$.
We break into cases: $\gamma <\varphi^\xi_0$ and 
 $\varphi^\xi_0\leq \gamma<\xi$.        In
 the first case $\iota(\gamma,\xi) = \iota(\gamma,\varphi^\xi_0)$
 and the claim follows by induction and  the fact that 
  $\operatorname{supp}(P_{\alpha,\varphi^\xi_0})= 
  \operatorname{supp}(P_{\alpha,\xi})$ for
  all $\alpha <\iota(\gamma,\xi)$. 
  Now consider $\varphi^\xi_0\leq \gamma <\xi$.
  If $\alpha<\iota_\xi$,
  then $P_{\alpha,\xi}=P_{\alpha,\varphi^\xi_0}$ and,
  since $\iota_\xi\leq \iota(\gamma{+}1,\xi)$,   
  $ \gamma$ is not in  $\operatorname{supp}(P_{\alpha,\varphi^\xi_0})$. 
 Otherwise, 
  choose $\ell,n\in\omega$  so that $\iota_\xi 
  \leq \alpha < \iota(\gamma{+}1,\xi) = \iota(\gamma{+}1,
  \varphi^\xi_{\ell}+n)$ 
  as in the definition
  of $\iota(\gamma,\xi)$.  By the minimality in the choice
  of $\varphi^\xi_\ell+n$, it follows that
   $\gamma$ is not in
    $\operatorname{supp}(P_{\alpha,
   \varphi^\xi_{\ell'}+n})$ for all $\ell',n\in\omega$. 
Since $\operatorname{supp}(P_{\alpha,\xi})$ is 
the union of all such sets, it follows that
 $\gamma$ is not in $
  \operatorname{supp}(P_{\alpha,\xi})$. 
   
    \medskip
    
    Next we prove, by induction on $\zeta$, that the matrix
    so chosen will additionally satisfy condition (1).    We first
    find a reformulation of condition (1).
    Note that by Claim (c), 
    $\mathcal A_{\alpha,\xi,\eta} = \{
    \dot a_\gamma : \gamma\in 
    E\cap \eta\setminus
         \operatorname{supp}(P_{ \alpha,\xi}) \}$.
         \medskip

\noindent Claim (e):  For 
each  $\alpha <\kappa$ and $\xi<\eta<\zeta$ and
finite subset $\{\gamma_i : i<m\}$
  of $ E\cap \eta\setminus
         \operatorname{supp}(P_{ \alpha,\xi})$   
         there is a $\beta<\kappa$ such that $\iota(\xi,\eta)\leq\beta$,
         $\{\gamma_i : i<m\}\subset
         \operatorname{supp}(P_{\beta,\eta})$
         and $P_{\beta,\eta}$ forces that 
         $\{ \dot a_{\gamma_i} : i<m\}$ is thin over
          the forcing extension by $P_{\alpha,\xi}$.
          
          \medskip
          
         Let us verify that Claim (e) follows from 
  condition (1).
           Let
  $\alpha,\xi,\eta$ 
          and $\{ \gamma_i : i<m\}$
          be as in the statement of  Claim (e).
Choose $\beta<\kappa$ 
      so that $\iota(\xi,\eta)$ and each $\iota(\gamma_i{+}1,\eta)$ 
      is less than     $\beta$.  Then $P_{\alpha,\xi}\cprec 
      P_{\beta,\eta}$ and $\{\dot a_{\gamma_i} : i<m\}
      \subset \mathcal A_{\alpha,\xi,\eta}\setminus
      \mathcal A_{\beta,\eta,\eta}$. This value of $\beta$
      satisfies the conclusion of   the Claim. 
     \medskip
      
      Now assume that Claim (e) holds and we prove
      that condition (1) holds.
 Assume 
      that $P_{\alpha,\xi}\cprec P_{\delta,\eta}$. To prove
      that  $\mathcal A_{\alpha,\xi,\eta}
      \setminus \mathcal A_{\delta,\eta,\eta}$ is forced by
      $P_{\delta,\eta}$ to be
      thin over
      the forcing extension by $P_{\alpha,\xi}$, it suffices
      to prove this for any finite subset of
      $\mathcal A_{\alpha,\xi,\eta}
      \setminus \mathcal A_{\delta,\eta,\eta}$. 
      Thus, let $\{\gamma_i : i<m\}$ be any finite
      subset of $\operatorname{supp}(P_{\delta,\eta})\cap 
      E\cap \eta\setminus
         \operatorname{supp}(P_{ \alpha,\xi})$. Choose
         $\beta $
as in the conclusion of    the Claim.  If $\beta \leq \delta$,
 then $P_{\delta,\eta}$  forces that 
        $\{\dot a_{\gamma_i} : i <m\}$ is thin over the forcing extension
        because $P_{\beta,\eta}\cprec P_{\delta,\eta}$ does.
        Similarly, if $\delta<\beta$, then $P_{\delta,\eta}$
        being completely embedded in $P_{\beta,\eta}$
        can not force that 
      $\{\dot a_{\gamma_i} : i<m\}$  is
   not thin over the forcing extension by $P_{\alpha,\xi}$.
 
\bigskip

    We assume
    that $\omega \leq \zeta<\kappa^+$ is a limit and that 
     $\langle P_{\alpha,\xi} : \alpha<\kappa, \xi<\zeta\rangle$ have been 
     chosen so that conditions (1)-(10) are satisfied. We prove,
     by induction on $n\in\omega$,
     that there is an extension 
     $\langle P_{\alpha,\xi} : \alpha<\kappa, \xi<\zeta+n\rangle$ 
     that also satisfies conditions (1)-(10). 
\bigskip

     For $n=1$, 
 we define the sequence
 $\langle   P_{\alpha,\zeta} : \alpha<\kappa \rangle$ according
 to the requirement of (3) or (4) as appropriate. 
   It
 follows from Lemma \ref{Lchains} that   (2) 
 will hold for the extension
  $\langle P_{\alpha,\xi} : \alpha<\kappa, \xi<\zeta+1\rangle$.
Conditions (3)-(10) hold since there are no new 
  requirements.  We must verify that the condition
  in  Claim (e) holds for $\eta=\zeta$.
  Let
  $\alpha,\xi $ 
          and $\{ \gamma_i : i<m\}$
          be as in the statement of  Claim (e) with $\eta=\zeta$.
          Let $C_\zeta = \{ \eta_\beta : \beta < o(C_\zeta)\}$ 
          be an order-preserving enumeration. 
          We first deal with case that 
$\acc(C_\zeta)$ is cofinal in $\zeta$.
          Choose any $\beta_0<\kappa$  
 large enough so that $\gamma_i \in \operatorname{supp}(P_{
 \beta_0,\zeta})$ for all $i<m$. Choose $\beta_0<\beta$
 so that $\iota(\xi, \eta_{\beta_0})\leq \beta$. Now we
 have that $P_{\alpha,\xi}\cprec P_{\beta,\eta_{\beta_0}}$ 
 and $P_{\beta,\eta_{\beta_0}}\cprec P_{\beta,\zeta}$. Applying
 Claim (e) to $\eta_{\beta_0}$, we have that $P_{\beta,\eta_{\beta_0}}$
 forces that $\{ \dot a_{\gamma_i} : i<m\}$ is thin over
 the forcing extension by $P_{\alpha,\xi}$. As in the proof of Claim
 (e), this implies that $P_{\beta,\zeta}$ forces the same thing. 
 
 Now assume that $\acc(C_\zeta)$ is not cofinal in $\zeta$. 
 If $\alpha  <\iota_\zeta$, then  apply Claim (e) to
 choose $\beta$ so that 
  $
 P_{\beta,\iota_\zeta}$ forces that $\{\dot a_{\gamma_i} :
 i<m\}$ is not thin over the extension by $P_{\alpha,\xi}$.
 Since $P_{\beta,\iota_\zeta}\cprec P_{\beta,\zeta}$ holds
 for all $\beta$, $P_{\beta,\zeta}$ also
 forces that $\{\dot a_{\gamma_i}  : i<m\}$ is not thin
 over the extension by $P_{\alpha,\xi}$.
  If $\iota_\zeta \leq \alpha$,  first choose
  $\delta<\kappa$ large enough so that
  $\iota(\xi,\zeta)$ and each $\iota(\gamma_i{+}1,\zeta)$
  is less than $\delta$. Since $\{ \gamma_i : i<m\}$
  is a subset of $\operatorname{supp}(P_{\delta,\zeta})$,
   we can  choose $\ell<\omega$ large enough so that
  $\{\gamma_i : i<\omega\} \subset \operatorname
  {supp}(P_{\delta,\varphi^\zeta_\ell})$. Applying
  Claim (e) to $\eta=\varphi^\zeta_\ell$, we choose
  $\beta$ as in the Claim.  As we have seen, 
  there is no loss to assuming that $\delta\leq \beta$
  and, since $P_{\beta,\varphi^\zeta_\ell}\cprec
   P_{\beta,\zeta}$, this completes the proof.
  
          \medskip
    
    If $o(C_\zeta)<\kappa$, then the construction
    of $\langle P_{\alpha,\zeta+n} : n\in\omega, \alpha<\kappa\rangle$
    is canonical so that conditions (2)-(10) hold. We again verify
    that Claim (e) holds for all values of $\eta$ with $\zeta<\eta<\zeta{+}\omega$. 
    Let $\alpha,\xi$ and $\{ \gamma_i : i<m\}$ be as in Claim (e)
    for $\eta=\zeta{+}n$. 
    We may assume
    that assume that $\{\gamma_i : i<m\}\cap \zeta 
     = \{\gamma_i : i<\bar m\}$ for some $\bar m\leq m$.
 If $\xi<\zeta$, let $\bar\xi=\xi$, otherwise, choose any
  $\bar\xi<\zeta$ so that $P_{\alpha,\zeta} = P_{\alpha,\bar\xi}$. 
  Note that $\{\gamma_i : \bar m\leq i<m\}$ is disjoint from
  the interval $[\zeta,\xi)$.
   Choose
     $\beta<\kappa$ to be greater than $\iota(\bar\xi,\zeta)$ and 
     each $\iota(\gamma_i{+}1,\zeta)$ ($i<\bar m$), and so
     that $P_{\beta,\zeta}$ forces that $\{ \dot a_{\gamma_i} : i<\bar m\}$
     is thin over the extension by $P_{\alpha,\bar\xi}$. 
     If $\bar m= m$ we are done by the fact that
      $P_{\alpha,\xi} $ is isomorphic to
       $P_{\alpha,\bar\xi}*\mathcal C_\omega$.
       In fact, we similarly have that $P_{\beta,\xi}$ forces
       that $\{ \dot a_{\gamma_i}  : i<\bar m\}$ is thin over
       the forcing extension by $P_{\alpha,\xi}$. 
Since   $P_{\beta,\zeta{+}n}$ forces that $\bigcup
      \{ \dot a_{\gamma_i} : \bar m \leq i < m\}$ is a Cohen real
      over the forcing extension by $P_{\beta,\xi}$ 
      it also follows that $P_{\beta,\zeta{+}n}$  
      forces that $\{ \dot a_{\gamma_i} : i<m\}$
      is thin over the extension by $P_{\alpha,\xi}$.  
\bigskip

Now we come to the final case where $o(C_\zeta)=\kappa$
and the main step to the proof. The fact that Claim (e)
will hold for $\eta=\zeta+1$ is proven as above for
the case when $o(C_\zeta)<\kappa$ and 
$\acc(C_\zeta)$ is
cofinal in $\zeta$. For values of $n>3$, there is nothing
to prove since $P_{\alpha,\zeta{+}3{+}k} = P_{\alpha,\zeta{+}3}$
for all $k\in\omega$. We also note that $\zeta{+}n\notin
E$  for all $n\in\omega$.

At step $\eta=\zeta+2$ we         must
take great care to preserve Claim (e) and at step $\zeta+3$
we make a strategic choice towards ensuring that 
 $\mathfrak p$ will equal $\kappa$. Indeed, 
         we begin by choosing
         the lexicographic
         minimal pair,         $(\xi_\zeta,\alpha_\zeta) $,
 in         $ \zeta\ttimes \kappa$ with the property
 that there is a family of fewer than $\kappa$
         many canonical
         $P_{\alpha_\zeta,\xi_\zeta}$-names of subsets of $\omega$
         and a $p\in P_{\alpha_\zeta,\xi_\zeta}$ that 
         forces over $P_{\kappa,\zeta}$ 
         that there is no  
         pseudo-intersection. 
         If there is no such pair, then let $(\alpha_\zeta,
         \xi_\zeta) = (\omega,\zeta{+}1)$. 
         Choose $\iota_\zeta$ so that $P_{\alpha_\zeta,\xi_\zeta}
         \cprec P_{\iota_\zeta,\zeta{+}1}$.
           
         \medskip
    
    Assume that $\alpha,\xi, \{\gamma_i : i < m\}$ are as in Claim (e).    
    We first
    check that if $\xi <\zeta+2$, then there is nothing new to 
    prove. Indeed, simply choose
    $\beta<\kappa$ large enough
    so that $P_{\beta,\zeta+1}$ has the properties
    required in Claim (e) for $P_{\alpha,\xi}$. 
    Of course it follows that
     $P_{\beta,\zeta+2}$ forces that $\{
     \dot a_{\gamma_i} : i<m\}$ is
     thin over the extension by $P_{\alpha,\xi}$
 since $P_{\beta,\zeta+1}$ already forces this.
 
 This means that we need only consider instances
 of Claim (e) in which $\xi=\zeta+2$.  The analogous statement
 also holds when we move to $\zeta+3$.     
    For each $\beta<\kappa$,
   let $$T_\beta= 
   E\cap  \operatorname{supp}(P_{\beta+1,  \zeta}) 
    \setminus  \operatorname{supp}(P_{\beta,\zeta})~$$
    and note that $P_{\beta{+}1,\zeta{+}1}$ forces
    that $\{ \dot a_\gamma : \gamma\in T_\beta\}$ is
    thin over the extension by $P_{\beta,\zeta{+}1}$. 
    Most of the work has been done for us in Lemma \ref{usingL}. 
    Except for some minor re-indexing, we can assume
    that the sequence $\{P_\beta : \beta<\kappa\}$ in
    the statement of Lemma \ref{usingL} is the
    sequence $\{ P_{\beta,\zeta} :   \beta < \kappa\}$.
    We also have that $P_{\beta,\zeta}*\mathcal C_{(\beta{+}1)\ttimes 
    \mathfrak c}$ is isomorphic to $P_{\beta,\zeta{+}1}$. 
    We can choose any $P_{0,\zeta{+}1}$-name
     $\dot{\mathcal D}_{0,\zeta}$-name of a Ramsey
     ultrafilter on $\omega$. The family 
      $\{ \dot a_\gamma : \gamma \in T_\beta\}$ will play
      the role of $\dot{\mathcal A}_\beta$ in the statement
      of Lemma \ref{usingL}, and we let
       $\{ \dot{\mathcal D}_{\beta,\zeta} : 0<\beta < \kappa\}$
       be the sequence as supplied in Lemma \ref{usingL}.

       Now assume that $\alpha < \kappa$ and that
        $\{\gamma_i : i<m\} \subset E\cap \zeta \setminus
         \operatorname{supp}(P_{\alpha,\zeta{+}1})$.
         Let $\{ \dot F_\ell : \ell\in \omega\}$ be any
         sequence of $P_{\alpha,\zeta+2}$-names
         of pairwise disjoint elements of $[\omega]^k$ for
         some $k\in\omega$. We must find a sufficiently
         large $\beta<\kappa$ so that $P_{\beta,\zeta+2}$
         forces that $\dot a_{\gamma_0}\cup
         \cdots \cup \dot a_{\gamma_{m-1}}$ is disjoint
         from $\dot F_\ell$ for some $\ell\in\omega$. 
         Let $\{ \beta_j : j <\bar m\}$ be the set (listed
         in increasing order) of
          $\beta<\kappa$ such that $T_\beta\cap
          \{\gamma_i : i<m\}$ is not empty
          and let $\beta_m = \beta_{m{-}1}+1$.
          By re-indexing
          we can assume  there is a sequence
           $\{ m_j  : j\leq \bar m\}\subset m{+}1$ so that
            $\gamma_i \in T_{\beta_j}$ for $ m_j\leq i < m_{j+1}$.
    Although $P_{\beta,\zeta{+}2} = P_{\beta,\zeta{+}1}$ for
    values of $\beta < \iota_\zeta$,  
    we will let $\bar P_{\beta,\zeta{+}2} = P_{\beta,\zeta{+}1}
    *_{\zeta{+}1} \mathbb L(\dot{\mathcal D}_{\beta,\zeta})$ for
     $\beta < \iota_\zeta$, and for consistent notation,
      let $\bar P_{\beta,\zeta{+}2}  = P_{\beta,\zeta{+}2}$ for
       $\iota_\zeta\leq\beta <\kappa$. We note
       that  $\{ \dot F_\ell : \ell\in \omega\}$ is also  
         sequence of $\bar P_{\alpha,\zeta+2}$-names
         of pairwise disjoint elements of $[\omega]^k$. 
       
          For each $j<\bar m$, let $\dot L_{j+1}$
          be the $\bar P_{\beta_j{+}1,\zeta+2}$-name of
          those $\ell$ such that $\dot F_\ell$ is disjoint from
           $\bigcup \{ \dot a_{\gamma_i}  : i<m_{j{+}1}\}$.
           It follows, by induction on $j<\bar m$,
            that  $\bar P_{\beta_j{+}1,\zeta{+}2}$ forces
            that 
            $\dot L_{j+1}$ is infinite since $\bar P_{\beta_j{+}1,\zeta{+}2}$
            forces that $\{\dot a_{\gamma_i} : m_j\leq i < m_{j{+}1}\}$
            is thin over the forcing extension by
             $\bar P_{\beta_j,\zeta{+}2}$. It now follows
             $\bar P_{\beta_m, \zeta {+}2}$ forces
             that $\{ \dot a_{\gamma_i} : i< m\}$ is thin
             over the forcing extension by $\bar P_{\alpha,\zeta{+}2}$.
 If $\beta_m< \iota_\zeta$, let $\beta = \iota_\zeta$, otherwise,
  let $\beta = \beta_m$. It follows that $P_{\beta,\zeta{+}2}$
  forces that $\{ \dot a_{\gamma_i} : i<m\}$ is thin over
  the forcing extension by $ P_{\alpha,\zeta{+}2}\cprec
   \bar P_{\alpha,\zeta{+}2}$.
   This completes the verification of Claim (e) for the case
    $\eta = \zeta{+}2$ and we now turn to the final case
    of $\eta = \zeta{+}3$. 
    
    We have chosen the pair $(\alpha_\zeta,\xi_\zeta)$ when
     choosing $\iota_\zeta$.  Let $\dot Q_{\iota_\zeta,\zeta{+}2}$
     be the $P_{\iota_\zeta,\zeta{+}2}$-name of the 
     finite support product of
     all posets of the form $\mathbb M(\mathcal F)$ 
     where $\mathcal F$ is a family of fewer than 
     $\kappa$ canonical
      $P_{\alpha_\zeta,\xi_\zeta}$-names of subsets
      of $\omega$ that is forced to have the sfip. Since
       $P_{\alpha_\zeta,\xi_\zeta}\in H(\lambda^+)$ 
       the set of all such families $\mathcal F$ is an
       element of $H(\lambda^+)$.  This 
       is our value of $\dot Q_{\iota_\zeta,\zeta{+}2}$
       as in condition (9) for the
       definition of $P_{\beta,\zeta{+}3}$  
       for all $\beta<\kappa$. The fact that Claim (e)
       holds in this case follows immediately from
        the induction hypothesis and Proposition \ref{nullstep}. 
     We also note that $P_{\iota_\zeta,\zeta{+}3}$ forces
     that every family of fewer than $\kappa$
     many canonical $P_{\alpha_\zeta,\xi_\zeta}$-names
     that is forced to have the sfip  is also forced,
     by $P_{\kappa,\zeta{+}3}$ to have 
     a pseudo-intersection. This means that for values
     of $\zeta'>\zeta$ with $o(\acc(C_\zeta))=\kappa$,
       the pair $(\alpha_\zeta,\xi_\zeta)$
     will be lexicographically strictly smaller than the choice
     for $\zeta'$. In other words,
      the family $\{ (\xi_\zeta , \alpha_\zeta) : \zeta<\kappa^+,
       \operatorname{cf}(\zeta)=\kappa\}$ is 
       strictly increasing in the lexicographic ordering.
     
     Now we can verify that $P_{\kappa,\kappa^+}$ forces
     that $\mathfrak p \geq \kappa$.  
 If it does not, then there is  a $\delta<\kappa$
 and a family,   
   $\{ \dot y_\gamma : \gamma < \delta\}$ 
      of canonical $P_{\kappa,\kappa^+}$-names of subsets
      of $\omega$ with some $p\in P_{\kappa,\kappa^+}$
      forcing that the family has sfip
      but
      has no pseudo-intersection. 
      By an easy modification of the names, we
      can assume that every condition in $P_{\kappa,\kappa^+}$
      forces that 
      the family
       $\{ \dot y_\gamma : \gamma<\delta\}$ is forced to have sfip.
 Choose any $\xi<\kappa^+$ so that $p\in P_{\kappa,\xi}$
 and 
 every $\dot y_\gamma$ is a $P_{\kappa,\xi}$-name.
Choose 
 $\alpha <\kappa$ large enough 
 so that $p\in P_{\alpha,\xi}$,
 $\iota(\bar\zeta,\xi) $,
 and each $\alpha_\gamma$ ($\gamma<\delta$) is less
 than $\alpha$. It follows that $\dot y_\gamma$ is
 a $P_{\alpha,\xi}$-name for all $\gamma<\delta$. 
 Since the family $\{ (\xi_\zeta , \alpha_\zeta) : \zeta<\kappa^+,
       \operatorname{cf}(\zeta)=\kappa\}$ is 
       strictly increasing in the lexicographic ordering,
       and this ordering on $\kappa^+\ttimes\kappa$ has order
       type $\kappa^+$, there is a minimal
        $\zeta<\kappa^+$ (with $\operatorname{cf}(\zeta)=\kappa$)
        such that $(\xi,\alpha)\leq (\xi_\zeta,\alpha_\xi)$. 
        By the assumption on $(\alpha,\xi)$, $(\xi_\zeta,\alpha_\xi)$
      will be chosen to equal $(\xi,\alpha)$.   
 One of the factors of the poset $\dot Q_{\iota_\zeta,\zeta{+}2}$
 will be chosen to be $\mathbb M(\{\dot y_\gamma : \gamma<\delta\})$.
 This proves that $P_{\kappa,\zeta{+}3}$ forces 
  $\{ \dot y_\gamma : \gamma<\delta\}$ does have a 
  pseudo-intersection.
  
  It should be clear from condition (8) in the construction
  that $P_{\kappa,\kappa^+}$ forces that $\mathfrak s\geq\kappa^+$.
    To finish the proof we must show that $P_{\kappa,\kappa^+}$
    forces that $\{ \dot a_\gamma : \gamma\in E\}$ is
     $(\kappa,\kappa^+)$-shattering. Since $\dot a_{\gamma}$
     is forced to be a Cohen real over the extension by 
      $P_{\kappa,\gamma}$, condition (2) in the Definition
      \ref{kappalambda}
      of $(\kappa,\kappa^+)$-shattering holds. Finally,
       we verify condition (1) of Definition \ref{kappalambda}. Choose
       any $P_{\kappa,\kappa^+}$-name $\dot b$ of an infinite subset
       of $\omega$. Choose any $(\alpha,\xi)\in \kappa\ttimes\kappa^+$
       so that $\dot b$ is a $P_{\alpha,\xi}$-name. The
       set $E\cap \operatorname{supp}(P_{\alpha,\xi})$ 
       has cardinality less than $\kappa$. For any $\gamma\in E\setminus
       \operatorname{supp}(P_{\alpha,\xi})$, there is a $(\beta,\zeta)\in 
        \kappa\ttimes \kappa^+$ such that $\{\dot a_\gamma\}$
        is thin over the forcing extension by $P_{\alpha,\xi}$. 
        It follows trivially that $P_{\beta,\zeta}$ forces
        that $\dot b$ is not a (mod finite) subset of $\dot
         a_\gamma$. 
 \end{proof}

\section{Questions}

\begin{enumerate}
\item Is it consistent to have $\omega_1<\mathfrak h < \mathfrak b <
  \mathfrak s$ and $\mathfrak{c}$ regular? 
\item  Is it consistent to have $\omega_1 < \mathfrak h < \mathfrak s
  < \mathfrak b$ and $\mathfrak{c}$ regular?
\end{enumerate}


\begin{bibdiv}

\def\cprime{$'$} 

\begin{biblist}


	
	
\bib{BPS80}{article}{
   author={Balcar, Bohuslav},
   author={Pelant, Jan},
   author={Simon, Petr},
   title={The space of ultrafilters on ${\bf N}$ covered by nowhere dense
   sets},
   journal={Fund. Math.},
   volume={110},
   date={1980},
   number={1},
   pages={11--24},
   issn={0016-2736},
   review={\MR{600576}},
   doi={10.4064/fm-110-1-11-24},
}
	
\bib{BaumDordal}{article}{
   author={Baumgartner, James E.},
   author={Dordal, Peter},
   title={Adjoining dominating functions},
   journal={J. Symbolic Logic},
   volume={50},
   date={1985},
   number={1},
   pages={94--101},
   issn={0022-4812},
   review={\MR{780528 (86i:03064)}},
   doi={10.2307/2273792},
}
\bib{Bla89}{article}{
   author={Blass, Andreas},
   title={Applications of superperfect forcing and its relatives},
   conference={
      title={Set theory and its applications},
      address={Toronto, ON},
      date={1987},
   },
   book={
      series={Lecture Notes in Math.},
      volume={1401},
      publisher={Springer, Berlin},
   },
   date={1989},
   pages={18--40},
   review={\MR{1031763}},
   doi={10.1007/BFb0097329},
}
\bib{BlassShelah}{article}{
   author={Blass, Andreas},
   author={Shelah, Saharon},
   title={Ultrafilters with small generating sets},
   journal={Israel J. Math.},
   volume={65},
   date={1989},
   number={3},
   pages={259--271},
   issn={0021-2172},
   review={\MR{1005010 (90e:03057)}},
   doi={10.1007/BF02764864},
}	
	
\bib{BrendleFischer}{article}{
   author={Brendle, J{\"o}rg},
   author={Fischer, Vera},
   title={Mad families, splitting families and large continuum},
   journal={J. Symbolic Logic},
   volume={76},
   date={2011},
   number={1},
   pages={198--208},
   issn={0022-4812},
   review={\MR{2791343 (2012d:03113)}},
   doi={10.2178/jsl/1294170995},
}

\bib{BrendleRaghavan}{article}{
   author={Brendle, J{\"o}rg},
   author={Raghavan, Dilip},
   title={Bounding, splitting, and almost disjointness},
   journal={Ann. Pure Appl. Logic},
   volume={165},
   date={2014},
   number={2},
   pages={631--651},
   issn={0168-0072},
   review={\MR{3129732}},
   doi={10.1016/j.apal.2013.09.002},
}

 

\bib{vDHandbook}{collection}{
   title={Handbook of set-theoretic topology},
   editor={Kunen, Kenneth},
   editor={Vaughan, Jerry E.},
   publisher={North-Holland Publishing Co., Amsterdam},
   date={1984},
   pages={vii+1273},
   isbn={0-444-86580-2},
   review={\MR{776619 (85k:54001)}},
}
	\bib{ssingular}{article}{
   author={Dow, Alan},
   author={Shelah, Saharon},
   title={On the cofinality of the splitting number},
   journal={Indag. Math. (N.S.)},
   volume={29},
   date={2018},
   number={1},
   pages={382--395},
   issn={0019-3577},
   review={\MR{3739621}},
   doi={10.1016/j.indag.2017.01.010},
}
\bib{Shelahpseudo}{article}{
   author={Dow, Alan},
   author={Shelah, Saharon},
   title={Pseudo P-points and splitting number},
   journal={Arch. Math. Logic},
   volume={58},
   date={2019},
   number={7-8},
   pages={1005--1027},
   issn={0933-5846},
   review={\MR{4003647}},
   doi={10.1007/s00153-019-00674-x},
}

\bib{Fischer18}{article}{
   author={Fischer, Vera},
   author={Friedman, Sy D.},
   author={Mej\'{\i}a, Diego A.},
   author={Montoya, Diana C.},
   title={Coherent systems of finite support iterations},
   journal={J. Symb. Log.},
   volume={83},
   date={2018},
   number={1},
   pages={208--236},
   issn={0022-4812},
   review={\MR{3796283}},
   doi={10.1017/jsl.2017.20},
}
		
\bib{Fischer17}{article}{
   author={Fischer, Vera},
   author={Mejia, Diego Alejandro},
   title={Splitting, bounding, and almost disjointness can be quite
   different},
   journal={Canad. J. Math.},
   volume={69},
   date={2017},
   number={3},
   pages={502--531},
   issn={0008-414X},
   review={\MR{3679685}},
   doi={10.4153/CJM-2016-021-8},
}

\bib{FischerSteprans}{article}{
   author={Fischer, Vera},
   author={Stepr{\=a}ns, Juris},
   title={The consistency of $\germ b=\kappa$ and $\germ s=\kappa^+$},
   journal={Fund. Math.},
   volume={201},
   date={2008},
   number={3},
   pages={283--293},
   issn={0016-2736},
   review={\MR{2457482 (2009j:03078)}},
   doi={10.4064/fm201-3-5},
}
 \bib{SouslinForcing}{article}{
   author={Ihoda, Jaime I.},
   author={Shelah, Saharon},
   title={Souslin forcing},
   journal={J. Symbolic Logic},
   volume={53},
   date={1988},
   number={4},
   pages={1188--1207},
   issn={0022-4812},
   review={\MR{973109}},
   doi={10.2307/2274613},
}
 
 \bib{Jech2}{book}{
   author={Jech, Thomas},
   title={Set theory},
   series={Springer Monographs in Mathematics},
   note={The third millennium edition, revised and expanded},
   publisher={Springer-Verlag, Berlin},
   date={2003},
   pages={xiv+769},
   isbn={3-540-44085-2},
   review={\MR{1940513 (2004g:03071)}},
}

\bib{Diego13}{article}{
   author={Mej\'{\i}a, Diego Alejandro},
   title={Matrix iterations and Cichon's diagram},
   journal={Arch. Math. Logic},
   volume={52},
   date={2013},
   number={3-4},
   pages={261--278},
   issn={0933-5846},
   review={\MR{3047455}},
   doi={10.1007/s00153-012-0315-6},
}
 
\bib{Boulder}{article}{
   author={Shelah, Saharon},
   title={On cardinal invariants of the continuum},
   conference={
      title={Axiomatic set theory},
      address={Boulder, Colo.},
      date={1983},
   },
   book={
      series={Contemp. Math.},
      volume={31},
      publisher={Amer. Math. Soc., Providence, RI},
   },
   date={1984},
   pages={183--207},
   review={\MR{763901 (86b:03064)}},
   doi={10.1090/conm/031/763901},
}

\bib{Shelah2011}{article}{
   author={Shelah, Saharon},
   title={The character spectrum of $\beta(\Bbb N)$},
   journal={Topology Appl.},
   volume={158},
   date={2011},
   number={18},
   pages={2535--2555},
   issn={0166-8641},
   review={\MR{2847327}},
   doi={10.1016/j.topol.2011.08.014},
}
		

\end{biblist}
\end{bibdiv}
\end{document}